\documentclass{amsart}
\usepackage[utf8]{inputenc}
\usepackage{color}
\usepackage[english]{babel}
\usepackage{amsmath} 
\usepackage{amsfonts}
\usepackage{amssymb}
\usepackage{setspace}
\usepackage[utf8]{inputenc}
\usepackage[T1]{fontenc}
\usepackage{graphicx}
\usepackage{amsthm}
\usepackage{mathrsfs}
\usepackage{bm}
\usepackage[all]{xy}
\usepackage{lscape}
\usepackage{latexsym}
\usepackage{epigraph}
\usepackage[Sonny]{fncychap}
\usepackage{amsfonts}

\theoremstyle{plain}
\newtheorem{theorem}{Theorem}
\newtheorem{proposition}{Proposition}
\newtheorem{lemma}{Lemma}
\newtheorem{corollary}{Corollary}

\theoremstyle{definition}
\newtheorem{definition}{Definition}

\theoremstyle{remark}
\newtheorem{remark}{Remark}
\definecolor{darkred}{rgb}{1,0,0} 
\definecolor{darkgreen}{rgb}{0,1,0}
\definecolor{darkblue}{rgb}{0,0,1}

\newcommand{\A}{\mathbb{A}}
\newcommand{\K}{\mathfrak{k}}
\newcommand{\B}{\mathcal{B}}

\definecolor{darkred}{rgb}{1,0,0} 
\definecolor{darkgreen}{rgb}{0,0.8,0}
\definecolor{darkblue}{rgb}{0,0,1}

\begin{document}

\title{Meridional rank of knots whose exterior is a graph manifold} 

\author{Michel Boileau}\thanks{First author partially supported by ANR
projects 12-BS01-0003-01 and 12-BS01-0004-01\\ Second author supported by CAPES  grant 13522-13-2}
\address{Aix-Marseille Universit\'e, CNRS, Centrale Marseille, I2M, UMR 7373,
13453 Marseille, France}
\email{michel.boileau@univ-amu.fr}

\author{Ederson Dutra}
\address{Mathematisches Seminar, Christian-Albrechts-Universität zu Kiel, Ludewig-Meyn Str.~4, 24098 Kiel,
Germany}
\email{dutra@math.uni-kiel.de}

\author{Yeonhee Jang}
\address{Department of Mathematics, Nara Women's University, Nara, 630-8506 Japan}
\email{yeonheejang@cc.nara-wu.ac.jp}

\author{Richard Weidmann}
\address{Mathematisches Seminar, Christian-Albrechts-Universität zu Kiel, Ludewig-Meyn Str.~4, 24098 Kiel,
Germany}
\email{weidmann@math.uni-kiel.de}

\begin{abstract} We prove for a large  class of knots that the meridional rank coincides with the bridge number. This class contains all knots whose exterior is a graph manifold. This gives a partial answer to a question of S. Cappell and J. Shaneson \cite[pb 1.11]{Kirby}.
\end{abstract}

\maketitle

Let $\K$ be a knot in $S^3$. It is well-known that the knot group   of $\K$ can be generated by   $b(\K)$ conjugates of the meridian where $b(\K)$ is the bridge number   of $\K$. The meridional rank $w(\K)$ of $\K$ is the smallest number of conjugates of the meridian that generate its  group. Thus we always have $w(\K)\le b(\K)$. It was asked by S. Cappell and J. Shaneson \cite[pb 1.11]{Kirby}, as well as by K. Murasugi,  whether the opposite inequality always holds, i.e. whether $b(\K)=w(\K)$ for any knot $\K$. To this day no counterexamples are known but the equality has been verified in a number of cases: 
\begin{enumerate}
\item For generalized Montesinos knots this is due to Boileau and Zieschang \cite{Boi3}.
\item For torus knots this is a result of Rost and Zieschang \cite{RZ}.
\item  The case of knots of meridional rank $2$ (and therefore also knots with bridge number $2$) is due to Boileau and Zimmermann~\cite{Boi6}.
\item For a class of knots also refered to as generalized Montesinos knots, the equality is due to Lustig and Moriah \cite{LM}.
\item For some iterated cable knots this is due to Cornwell and Hemminger \cite{CH}.
\item For knots of meridional rank $3$ whose double branched cover is a graph manifold the equality can be found in \cite{BJW}.
\end{enumerate}


The \textit{knot space} of $\K$ is defined as $X(\K):=\overline{S^3-V(\K)}$ where $V(\K)$ is a regular neighborhood of $\K$ in $S^3$. The \textit{knot group} $G(\K)$ of $\K$ is the fundamental group of $X(\K)$. We further denote by $P(\K)\leq G(\K)$ the peripheral subgroup of $\K$, i.e., $P(\K)=\pi_1\partial X(\K)$. 

Let $m\in P({\K})$ be  the meridian of $\K$, i.e. an element of $G(\K)$ which can be represented    by a simple closed curve on $\partial X(\K)$ that bounds a disk in $S^3$ which intersects $\K$ in exactly one point.  In the sequel we refer to  any conjugate of $m$  as a meridian of $\K$.

We call a subgroup $U\leq G(\K)$ \textit{meridional}  if $U$ is generated by finitely many  meridians of $\K$. The minimal number of meridians needed to generated $U$, denoted by $w(U)$, is called the  \textit{meridional rank} of $U$.  Observe that the knot group $G(\K)$ is meridional and its meridional rank is equal to   $w(\K)$.

A meridional subgroup $U$ of meridional rank $w(U)=l$ is called \textit{tame} if for any $g\in G(\K)$  one of the following holds:
\begin{enumerate}
\item  $gP(\K)g^{-1}\cap U=1$.
\item $gP(\K)g^{-1}\cap U=g\langle m \rangle g^{-1}$ and there exists meridians $m'_2,\ldots,m'_l$ such that   $U$ is generated by $\{gmg^{-1},m'_2,\ldots,m'_l\}$.
\end{enumerate}


\begin{definition}
A non-trivial knot $\K$ in $S^3$ is called \textit{meridionally tame} if any meridional subgroup $U\leq G(\K)$ generated by less than $b(\K)$ meridians  is tame.
\end{definition} 

\begin{remark}
If $\K$ is  meridionally tame, then its group cannot be generated by less than $b(\K)$ meridians. Hence the bridge number equals the meridional rank. Thus the question of  Cappell and Shaneson has a positive answer for the class of meridionally tame knots by definition of meridional tameness.
\end{remark}

The class of meridionally tame knots trivially contains the class of 2-bridge knots as any cyclic meridional subgroup is obviously tame.  In Lemma \ref{lem:torus-knots} below   we show that the meridional tameness of torus knots is implicit in \cite{RZ} and in Proposition~\ref{3bridge} we show that prime 3-bridge knots are meridionally tame. However it follows from \cite{Ag} and the discussion at the end of Section~\ref{sec_3bridge} that satellite knots are in general not meridionally tame. For examples Whitehead doubles of  non-trivial knots are never meridionally tame. These are prime satellite knots, whose satellite patterns have winding number zero, which is opposite to the braid patterns considered in this article. Moreover Whitehead doubles of 2-bridge knots are prime 
4-bridge knots for which the question of  Cappell and Shaneson has a positive answer by \cite[Corollary 1.6]{BJW}. There exists also connected sums of meridionally tame knots (for examples some 2-bridge knots) which are not meridionally tame. In contrast, it should be noted that we do not know any hyperbolic knots that are not meridionally tame, but it is likely that such knots exist.

\medskip In this article we consider the class of knots $\mathcal K$ that is the smallest class of knots that contains all meridionally tame knots and is closed under connected sums and satellite constructions with braid patterns, see Section~\ref{knotsinC} for details. The following result is our main theorem:

\begin{theorem}\label{main} Let $\K$ be a knot  from $\mathcal K$. Then $w(\K)=b(\K)$.
\end{theorem}
 
As the only Seifert manifolds that can be embedded into $\mathbb R^3$ are torus knot complements, composing spaces and cable spaces (see \cite{JS}) and as cable spaces are special instances of braid patterns we immediately obtain the following consequence of Theorem~\ref{main}.
 
\begin{corollary}\label{CJSJ}
Let $\K$ be a knot such that its exterior is a graph manifold. Then $$w(\K)=b(\K).$$
\end{corollary}

In 2015  the first, third and fourth author found a proof that the meridional rank coincides with the bridge number for the easier case of knots obtained from torus knots by satellite operations with braid patterns. In the meantime the second and fourth authors introduced the notion of meridional tameness, and the second author was able to generalize the result to the broader class of knots 
$\mathcal K$ by proving Theorem~\ref{main}. Some basic features of the original proof are preserved, but the proof of Theorem~\ref{main} is much more involved and subtle,  in particular because of the presence of composing spaces.


\section{Description of the class $\mathcal K$. }\label{knotsinC}

In this section we introduce the appropriate formalism to study knots that lie in the class $\mathcal K$ introduced in the introduction. Recall that these knots are obtained from meridionally tame knots by repeatedly taking connected sums and performing satellite constructions with braid patterns.\color{black} 
 
We first  discuss satellite construction with braid patterns. This generalizes the well-known cabling construction. Satellite construction with braid patterns are also discussed in \cite{CH}.

\begin{figure}[h] 
\begin{center}
\includegraphics[scale=1]{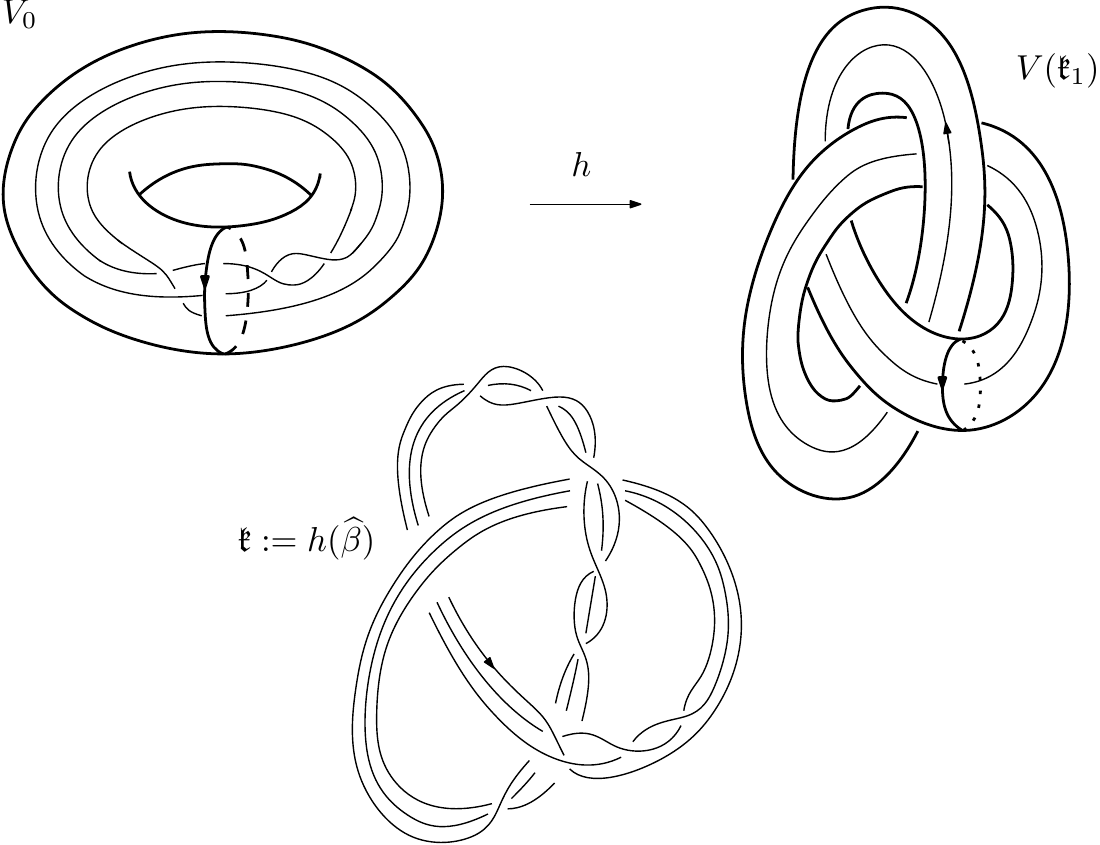}
\end{center}
\caption{$\beta$ is the $3$-braid $\sigma_2{\sigma_1}^{-1}\sigma_2^2$ and $\K_1$ is the trefoil knot.}\label{Fig1}
\end{figure}

\smallskip Let $V_0=\mathbb{D}^{2}\times S^1$ be a standardly embedded  solid torus in ${\mathbb{R}}^3$ and $\beta$ an $n$-braid with $n\geq 2$. Assume  that the closed braid $\widehat{\beta}$ is standardly  embedded in the interior of $V_0$,  see  Figure~\ref{Fig1}. For a  knot    $\K_1\subset S^3$ and a homeomorphism $h:V_0\rightarrow V(\K_1)$   onto a tubular neighborhood of   $\K_1$ which sends a meridian  of $V_0$ onto a meridian of $\K_1$ and the longitude of $V_0$ (i.e. the simple closed curve on $\partial V_0$ that is nullhomotopic in the complement of the interior of $V_0$) to the longitude of $\K$,  we  define the link $\K:=h(\widehat{\beta})$, see Figure \ref{Fig1}. Note that $\K$ is a knot if and only if  the associated permutation $\tau\in S_n$ of $\beta$ is a cycle of length $n$. In this case the knot $\K$ is called   a $\beta$-\textit{satellite} of $\K_1$ and $\K$ is denote by $\beta(\K_1)$. We also say that $\K$ is obtained from $\K_1$ by a \textit{satellite operation with braid pattern}  $\beta$ .

{\smallskip In order to define the class $\mathcal{K}$ we need some terminology.}

\smallskip Let $A$ be a finite tree. A subset $E\subset EA$ is called an \textit{orientation} of $A$ if $EA=E\overset{.}{\cup}E^{-1}$. For a finite rooted  tree $(A,v_0)$ we will define a natural orientation   determined by the root $v_0$ in the following way: for each vertex $v\in VA$ there exists a unique reduced path  $\gamma_v:=e_{v,1},\ldots,e_{v,r_v}$  in $A$  from $v_0$ to $v$. We define
$$E(A,v_0)=\{e_{v,i} \  \vert \  v\in VA, 1\leq i \leq r_v\}.$$  
Throughout this  paper we assume that any   rooted tree $(A,v_0)$  is endowed with this orientation.

\smallskip Suppose that $E$ is some orientation for $A$.   We say that a path $e_1,\ldots,e_k$ in $A$ is   \textit{oriented} if $e_i\in E$ for $1\le i\le k$. For any vertex  $v\in VA$    we define    $A(v)$ as the sub-tree of $A$ spanned by the set 
$$VA(v)=\{ \omega(\gamma) \ | \ \gamma \text{ is an oriented path in } A \text{  and } \alpha(\gamma)=v\}.$$

\smallskip Note that for any vertex $w\in VA$ we have  $$E(A(w),w)=E(A,v_0)\cap EA(w)$$ where we consider the canonical orientations defined above, see Figure~\ref{fig:rootedtree}.

\begin{figure}[h!]
\begin{center}
\includegraphics[scale=1]{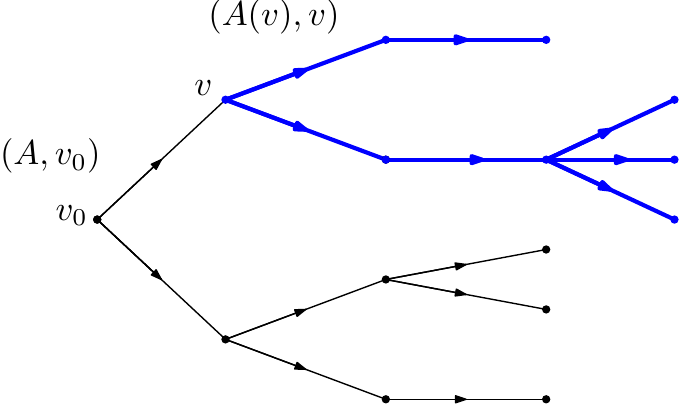}
\end{center}
\caption{A rooted tree $(A,v_0)$ and a rooted subtree $(A(w),w)$ (fat blue lines) with their canonical orientations. }\label{fig:rootedtree}
\end{figure}

Recall that for any vertex $v\in VA$, the \textit{star} of $v$ is defined as 
$$St(v,A)=\{e\in EA  \ | \ \alpha(e)=v\}.$$ The cardinality $|St(v,A)|$  of $St(v,A)$ is called the \textit{valence} of $v$, written $val(v,A)$. If $A$ is assigned some orientation $E\subset EA$  we  further  define the \textit{positive star} of $v$ as  $$St_{+}(v,A)=St (v,A)\cap E $$ and the \textit{positive valence } of $v$ as   $val_{+}(v,A)=|St_+(v,A)|$.

For $i\in \{0,1\}$ we set  
$$V_i:=\{v\in VA \ | \ val_{+}(v,A)=i\}$$
and we also define  
$$V_2:=\{v\in VA \ \vert \ val_{+}(v,A)\geq 2\}.$$  

We now define labelings of rooted trees.  A labeled rooted tree is a tuple 
$$\mathcal{A}=((A,v_0),\{\K_v \ | \ v\in V_0 \}, \{\beta_v \ | \ v\in V_1\})$$ 
such that the following hold:
\begin{enumerate}
\item $(A,v_0)$ is a finite rooted tree.
\item For any $v\in V_0$, $\K_v$ is a non-trivial knot.
\item For any $v\in V_1$,   $\beta_v$ is an   $n_v$-braid having $n_v\geq 2$ strands  such that the closed braid $\widehat{\beta_{v}}$ is a knot.
\end{enumerate}

For $w\in VA$ we define the labeled rooted tree $\mathcal{A}_w$ as
$$\mathcal{A}_w:=((A(w),w),\{\K_v \ | \ v\in V_0\cap VA(w)\},\{\beta_v \ | \ v\in V_1\cap VA(w)\}).$$

We now associate to any labeled rooted tree $\mathcal{A}$   a knot   $\K=\K_{\mathcal{A}}\subset S^3$. We define  $l(\mathcal A):=max\{ d(v,v_0) \ \vert \ v\in VA\}$. We recursively define  $\K_{\mathcal{A}}$ in the following way: 
\begin{enumerate} 
\item If $l(\mathcal A)=0$, then we define $\K_{\mathcal{A}}=\K_{v_0}$.
\item  If  $l(\mathcal A)>0$  and $val(v_0,A)=1$, then we define $\K_{\mathcal{A}}:=\beta_{v_0}(\K_{\mathcal{A}_{v_1}})$ where $v_1\in VA$ is the unique vertex of $A$  such that $d(v_1,v_0)=1$.
\item  If  $l(\mathcal A)>0$  and $val(v_0,A)\ge 2$  then  we define  $\K_{\mathcal{A}}:=\overset{d}{\underset{i=1}{\sharp}}\K_{\mathcal{A}_{v_i}}$  where $v_1,\ldots,v_d\in VA$ such that $d(v_i,v_0)=1$. 
\end{enumerate}
 
Note that this is a (recursive) definition indeed as in situation (2) and (3) we have $l(\mathcal {A}_{v_i})<l({\mathcal A})$ for all occurring $i$. Note moreover that $\K_{\mathcal A}$ lies in the class $\mathcal K$ if and only if $\K_v$ is meridionally tame for any $v\in V_0$. Thus we can rephrase the main theorem in the following way:

\begin{theorem}\label{main2} Let $\mathcal A$ be a labeled rooted tree and $ \K=\K_{\mathcal{A}}$. Suppose that $\K_v$  is meridionally tame for all $v\in V_0$. Then $$b(\K)=w(\K).$$
\end{theorem}

The proof relies on computing both the bridge number and the meridional rank. We conclude this section with the computation of the bridge number which is an easy consequence of the work of Schubert.

%

\smallskip Remember that any vertex $v\in V_1$ is labeled $\beta_v$, where  $\beta_v$ is a   braid with $n_v$ strands. We define  the function $n:VA\rightarrow \mathbb{N}$ by
$$
n(v)=\left\{
\begin{array}{lcl}
n_v & \text{if} & v\in V_1  \\
1 & \text{if}  & v\in V_0\cup V_2 \\
 \end{array}
\right.
$$
Recall that, for any $v\in VA$,  $\gamma_v= e_{v,1},\ldots,e_{v,r_v}$ is the unique reduced path in $A$ from  $v_0$ to $v$. We define the \emph{height}  of a vertex $v\in VA $ in $A$ as
$$ h(v):= \prod_{i=1}^{r_v} n({\alpha(e_{v,i})})$$
for  $v\neq v_0$ and  $h(v_0):=1$.



\begin{lemma}\label{bridgenumber}
Let $\K=\K_{\mathcal{A}}$ be the knot defined by the labeled rooted tree $\mathcal{A}$.  Then the bridge number of $\K$ is given by 
$$  b(\K)=\bigg[\sum_{v\in V_{0}}h(v)\cdot b({\K}_v)\bigg ]-\bigg[\sum_{v\in V_{2}}h(v)\cdot (val_{+}(v,A)-1)\bigg] $$
\end{lemma}
The proof of the Lemma relies in the following result of Schubert \cite{Schu}, see also  \cite{Schultens} for a more modern proof. Note that the result of Schubert is actually stronger than what we state. 
\begin{theorem}[Schubert]{\label{Sc}}
Let $\K_1,\ldots , \K_d$ be  knots in $S^3$ and $\beta$ be an $n$-braid such that the closed braid  $\widehat{\beta}$ is a knot. Then the following hold:
\begin{enumerate}
\item[(i)] $b(\K_1\sharp\ldots \sharp \K_d)=b(\K_1)+\ldots +b(\K_d)-(d-1)$.
\item[(ii)] $ b(\beta(\K_1))=n\cdot b(\K_1)$.
\end{enumerate}
\end{theorem}
\begin{proof}[Proof of Lemma~\ref{bridgenumber}]
The proof is by induction on  $l(\mathcal{A}):=max\{d(v,v_0) \ | \ v\in VA\}$. If $l(\mathcal{A})=0$, then $\K=\K_{v_0}$. Hence we have
$$ b(\K)=b(\K_{v_0})=\bigg[\sum_{v\in V_{0}=\{v_0\}}h(v)\cdot b({\K}_v)\bigg ]-\bigg[\sum_{v\in V_{2}=\emptyset}h(v)\cdot (val_{+}(v,A)-1)\bigg]$$ as $h(v_0)=1$.
 
Suppose that  $l(\mathcal{A})>0$ and $val_{+}(v_0,A)=1$. Let  $v_1\in VA$ be the unique vertex such that $d(v_0,v_1)=1$. In this case  $V_0\cup V_2 \subset VA(v_1)$. If $h_{v_1}$ denotes the height of a vertex in the rooted tree $(A(v_1),v_1)$, then it is easy to see that $h(v)=n({v_0})h_{v_1}(v)$ for all $v \in VA(v_1)$. Moreover, $val_{+}(v,A(v_1))=val_{+}(v,A)$ for all $v\in VA(v_1)$.  By Theorem \ref{Sc}(ii) and the induction hypothesis  we obtain:
\begin{eqnarray}
b(\K) & \overset{\text{3(ii)} }= & n_{v_0}\cdot b(\K_{\mathcal{A}_{v_1}})\nonumber\\
 & \overset{\text{i.h.} }= &   n_{v_0}\cdot \bigg[ \sum_{v\in V_0\cap VA(v_1)}h_{v_1}(v)\cdot b(\K_{v}) - \nonumber\\
 &   & -\sum_{v\in V_2\cap VA(v_1)}h_{v_1}(v)\cdot (val_+(v,A(v_1))-1) \bigg]\nonumber\\
 & = & \sum_{v\in V_0\cap VA(v_1)}n_{v_0}h_{v_1}(v)\cdot b(\K_{v}) -\sum_{v\in V_2\cap VA(v_1)}n_{v_0}h_{v_1}(v)\cdot (val_+(v,A)-1) \nonumber\\
& = &  \bigg[\sum_{v\in V_{0}}h(v)\cdot b({\K}_v)\bigg ]-\bigg[\sum_{v\in V_{2}}h(v)\cdot (val_{+}(v,A)-1)\bigg] \nonumber
\end{eqnarray}

Suppose now  that $l(\mathcal{A})>0$ and $d:=val_{+}(v_0,A)\ge 2$. Let $v_1,\ldots,v_d\in VA$ such that $d(v_i,v_0)=1$. By definition, $\K$ is equal to the connected sum $\overset{d}{\underset{i=1}{\sharp}}
{\K_{\mathcal{A}_{v_i}}}$. Observe that if   $h_{v_i}$ denotes the height  of a vertex in the rooted tree $(A(v_i),v_i)$, then $h_{v_i}(v)=h(v)$ for any $v\in VA(v_i)$.  By Theorem \ref{Sc}(i) and the induction hypothesis we obtain: 
\begin{eqnarray}
b(\K) & \overset{\text{3(i)} }= & \bigg[\sum_{i=1}^{d}{b(\K_{\mathcal{A}_{v_i}})}\bigg]-(d-1)=\bigg[\sum_{i=1}^{d}{b(\K(\mathcal{A}_{v_i}))}\bigg]-(val_{+}(v_0,A)-1)=\nonumber\\
      & \overset{\text{i.h.}}= & \bigg[\sum_{i=1}^{d}{\bigg( \sum_{v\in V_0\cap VA(v_i)}h_{v_i}(v)b({\K_{\mathcal{A}_{v_i}}}) -\sum_{v\in V_2\cap VA(v_i)}h_{v_i}(v)(val_{+}(v,A(v_i))-1) \bigg)}\bigg]-\nonumber\\
      &   &  -(val_{+}(v_0,A)-1) \nonumber\\
      & = & \bigg[\sum_{v\in V_{0}}h(v)\cdot b({\K}_v)\bigg ]-\bigg[\sum_{v\in V_{2}}h(w)\cdot (val_{+}(v,A)-1)\bigg]\nonumber
\end{eqnarray}
since $V_0=\underset{i}\cup(V_0\cap VA(v_i))$ and $V_2=\{v_0\}\cup (\underset{i}{\cup}(V_2\cup VA(v_i)))$. 
\end{proof}


\section{Description of $G(\K_{\mathcal{A}})$.}{\label{DSF}}

In the previous section we have constructed a knot $\K_{\mathcal A}$ from a labeled tree $\mathcal A$. The construction implies that the knot space $X(\K_{\mathcal A})$ contains a collection $\mathcal T$ of incompressible tori corresponding to the edges of $A$ such that to each vertex $v\in VA$ there corresponds a component of the complement of $\mathcal{T}$ such that the following hold:

\begin{enumerate}
\item The vertex space associated to each vertex $v\in V_0$ is $X(\K_v)$.
\item The vertex space associated to a vertex $v\in V_1$ is the braid space $CS(\beta_v)$, see below for details.
\item The vertex space associated to a vertex $v\in V_2$ is an $r$-fold composing space where $r=val_+(v,A)$.
\end{enumerate} 

Thus $X(\K_{\mathcal A})$ can be thought of as a tree of spaces. It follows from the theorem of Seifert and van Kampen that corresponding to the tree of spaces there exists a tree of groups decomposition $\mathbb A$ of  $G(\K_{\mathcal{A}})$ such that all edge groups are free Abelian of rank $2$. It is the aim of this section to describe this tree of groups. We will first describe the vertex groups that occur in this splitting and then conclude by describing the boundary monomorphisms of the tree of groups.

\textbf{(1)}  If $v\in V_0$ then the complementary component of $\mathcal T$ corresponding to $v$ is the knot space of $\K_v$. Thus we put $A_v:=G(\K_v)$.  Denote by  $m_v\in P(\K_v)$ the meridian  and by $l_v\in P(\K_v)$ the  longitude of $\K_v$.
 
\smallskip \textbf{(2)} We now describe the vertex group $A_v$ for $v\in V_1$. Let $n:=n_v$ be the number of strands of the associated braid $\beta_v$.  By definition the associated permutation  $\tau\in S_{n}$ of $\beta_v$ is an $n$-cycle (equivalently the closed braid $\widehat{\beta_v}$ is a knot) and  $\widehat{\beta_v}$ is standardly embedded in the interior of an unknotted  solid torus $V_0 \subset \mathbb{R}^3$, see Figure~\ref{Fig1}. The \textit{braid space of} $\beta_v$ is defined as 
$$CS(\beta_v):=\overline{V_0-V(\beta_v)}$$ 
where $V(\beta_v)$ is a regular neighborhood of $\widehat{\beta_v}$ contained in the interior of $V_0$, see Figure~\ref{Fig2}. The complementary component of $\mathcal T$ corresponding to $v$ is by construction homeomorphic to $CS(\beta_v)$.

There is an obvious fibration $CS(\beta_v)\rightarrow S^1$ of $CS(\beta_v)$ onto $S^1$ induced by the projection of $V_0=\mathbb{D}^2\times S^1$ onto the second factor. The fiber is clearly the space $$X:=\mathbb{D}^2-Q_{n},$$ where $Q_{n}$ is the union of the interior of $n$ disjoint   disks contained in the interior of the unit disk.  
\begin{figure}[h!]
\begin{center}
\includegraphics[scale=1]{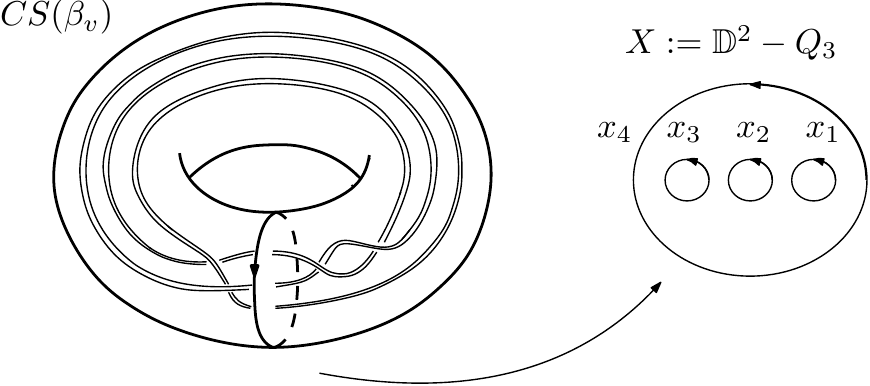}
\end{center}
\caption{The $3$-manifold $CS(\beta)$.}\label{Fig2}
\end{figure}

\smallskip We will denote the free generators of $\pi_1(X)$ corresponding to the boundary paths of the removed disks by $x_1,\ldots, x_{n}$. This gives a natural identification of $\pi_1(X)$ with $F_{n}$. We obtain 
the short exact sequence 
$$ 1\rightarrow F_n=\pi_1(X)\rightarrow \pi_1(CS(\beta_v))\rightarrow \pi_1(S^1)=\mathbb Z\rightarrow 1.$$
Let $t \in A_v:=\pi_1 CS(\beta_v)$ be the element represented by the loop in $CS(\beta_v)$ defined by $q_0\times S^1$, for a point $q_0\in \partial \mathbb{D}^2$.   Thus $t $ represensts a longitude of $V_0$ in the above sense.  We  can write $A_v$ as the semi-direct product $F_{n}\rtimes \mathbb{Z}$ where the  action of $\mathbb{Z}=\langle t \rangle$ on the fundamental group of the fiber is given by 
$$ tx_it^{-1}=A_ix_{\tau(i)}A_i^{-1} \ \ \ \ 1\leq i\leq n$$
and the words $A_1,\ldots, A_{n}\in F_{n}$ 
satisfy the identity
$$A_1x_{\tau(1)}A_1^{-1}\cdot\ldots\cdot A_{n}x_{\tau(n)}A_{n}^{-1}=x_1\cdot\ldots\cdot x_{n}$$ 
in the free group $F_{n}$.

Note that  any element of $A_v$ can be uniquely  written in the form $w\cdot t^r$ with $w\in F_{n}$ and $r\in \mathbb{Z}$. Moreover 
$$ C_v:=\pi_1(\partial V_0)=\langle x_{n+1},t\rangle$$
where $x_{n+1}:=x_1\cdot\ldots\cdot x_{n}$. Note that in a satellite construction with braid pattern $\beta_v$ the curve corresponding to $x_{n+1}$ is identified with the meridian of the companion knot and the curve corresponding to $t$ is identified with the longitude of the companion knot. Finally      
$$P_v:=\pi_1(\partial V(\beta_v))=\langle m_v,l_v\rangle$$
where  $m_v:=x_1$ and $l_v:=u \cdot t^n$ for some  $u \in F_{n}$. Note that for a satellite with braid patttern $\beta_v$, thus in particular for the knot $\K_{\mathcal A_v}$, $m_v$ and $l_v$ represent the meridian and the longitude of the satellite knot.

Note further that $F_{n}=\langle\langle m_v \rangle\rangle_{A_v}$ (normal closure in $A_v$) as any two elements of $\{x_1,\ldots ,x_{n}\}$ are conjugate in $A_v$. It follows in particular that $C_v\cap\langle\langle m_v\rangle\rangle_{A_v}=\langle x_{n+1}\rangle$.

\smallskip
\textbf{(3)} We now describe the vertex group $A_v$ if $v\in V_2$. Let $n=val_{+}(v,A)$.  By construction the complementary component of $\mathcal T$ corresponding to the vertex $v$ is homeomorphic to an $n$-fold composing space $W_{n}:=X\times S^1$, where $X=\mathbb{D}^2-Q_{n}$ is as before, see \cite{JS}. Thus
$$A_v:=\pi_1(W_{n})=\pi_1 (X)\times \pi_1(S^1)=\langle x_{1},\ldots,x_{n},t  \ | \ [x_{ i},t]=1\rangle.$$
Consequently any element of $A_v$ can be uniquely written as $w\cdot t^z$ with $w\in F_{n}$ and $z\in \mathbb{Z}$. Clearly $t$ generates the center of $A_v$.

\smallskip If the homeomorphism is chosen appropriately then we get the following with the above notation:
\begin{enumerate}
\item $P_v:=\langle l_v,m_v\rangle $ corresponds to the peripheral subgroup of $\K_{\mathcal A_v}$ where $l_v:=x_{n+1}$ is the longitude and $m_v:=t$ is the meridian.
\item There exists a bijection $j: St_{+}(v,A)\rightarrow \{1,\ldots ,n\}$ such that for any $e\in St_{+}(v,A)$ the subgroup  $C_e:=\langle l_e,m_v\rangle$ correponds to the peripheral subgroup of $\K_{\mathcal A_{\omega(e)}}$ with $l_e:=x_{j(e)}$ the longitude and $m_v:=t$ the meridian.  
\end{enumerate}

With the notation introduced we have  
$$A_v=F(\{l_e | e\in St_{+}(v,A)\})\times\langle m_v\rangle.$$
We further denote $F(\{l_e  |  e\in St_{+}(v,A)\}$ by $F_v$.

\textbf{(4)} For any edge $e\in E(A,v_0)$ the associated edge group $A_e$ is free Abelian generated by $\{m_e,l_e\}$.

\smallskip We now describe the boundary monomorphims.  For any $e\in E(A,v_0)$ with $v:=\alpha(e)$ and $w:=\omega(e)$ the boundary monomorphism   $\alpha_e:A_e\rightarrow A_{v}$ is given by:
$$
\alpha_e(m_e^{z_1}\cdot l_e^{z_2})=\left\{
\begin{array}{lcl}
x_{n+1}^{z_1}\cdot t^{z_2} & \text{if} & v\in V_1.  \\
m_{v}^{z_1}\cdot l_e^{z_2} & \text{if}  & v \in V_2. \\
 \end{array}
\right.
$$
while  $\omega_e:A_e\rightarrow A_{w}$   is given by $\omega_e(m_e^{z_1}\cdot l_e^{z_2})=m_w^{z_1}\cdot l_w^{z_2}$.

\section{Vertex Groups}

In this section we will study subgroups of the vertex groups $A_v$ of $\mathbb A$ for $v\in V_1\cup V_2$. As these vertex groups are semidirect products of a finitely generated free group and $\mathbb Z$ we start by considering certain subgroups of the free group $F_n$.
 
\medskip  We think of $F_n=F(x_1,\ldots,x_n)$ as the group given by the presentation 
$$\langle x_1,\ldots,x_{n+1}\mid x_1\cdot\ldots\cdot x_{n+1}\rangle$$ and identify $F_n$ with $\pi_1(X)$ where $X$ is the $(n+1)$-punctured sphere as before.

\smallskip We will study subgroups of $F_n$ that are generated by finitely many conjugates of the $x_i$. We call an element of $F_n$ {\em peripheral} if it is conjugate to $x_i^z$ for some $i\in\{1,\ldots ,n+1\}$ and $z\in\mathbb Z$.


\begin{lemma}\label{L2} Let $S=\{g_ix_{j_i}g_i^{-1}\mid 1\le i\le k\}$ with $k\geq 0$,  $g_i\in F_n$,  and $j_i\in\{1,\ldots ,n+1\}$ for $1\le i\le k$. Suppose that $U:=\langle S\rangle\neq F_n$.

Then there exists $T=\{h_lx_{p_l}h_l^{-1}\mid 1\le l\le m\}$ with $m\le k$ and $h_l\in F_n$ and $p_l\in\{1,\ldots ,n+1\}$ for $1\le l\le m$ such that the following hold:
\begin{enumerate}
\item $U$ is freely generated by  $T$.
\item  Any $h_lx_{p_l}h_l^{-1}$ is in U conjugate to some $g_ix_{j_i}g_i^{-1}$.
\item\label{con3} Any peripheral element of $U$ is in $U$ conjugate to an element of $\langle h_lx_{p_l}h_l\rangle$ for some $l\in\{1,\ldots ,m\}$. 
\end{enumerate}

In particular $\{j_1,\ldots,j_k\}=\{p_1,\ldots,p_m\}$.
\end{lemma}


\begin{remark} Note that the conclusion of Lemma~\ref{L2} does not need to hold if $U=F_n$. Indeed if $S=\{x_1,\ldots ,x_{n+1}\}\backslash\{x_i\}$ for some $i\in\{1,\ldots ,n+1\}$ then $\langle S\rangle=U=F_n$ and $U$ has $n+1$ conjugacy classes of peripheral subgroups, it follows that conclusion (\ref{con3}) cannot hold.
\end{remark}

 
\begin{proof} Let $\tilde X$ be the cover of $X$ corresponding to $U\le F_n=\pi_1(X)$. Note that the $U$-conjugacy classes of maximal peripheral subgroups of $U$ correspond to compact boundary components of $\tilde X$ that finitely cover boundary components of $X$. Note that for any $i$ the $U$-conjugacy class of  $\langle g_ix_{j_i}g_i^{-1}\rangle$ must correspond to a compact boundary component of $\tilde{X}$ for which the covering is of degree $1$. 

As $U$ is generated by peripheral elements it follows that the interior of $\tilde X$ is a punctured sphere. 
Note further that all but one punctures must correspond to compact boundary components of $\tilde{X}$ that cover components of $\partial X$ with degree $1$ and correspond to some $\langle g_ix_{j_i}g_i^{-1}\rangle$, otherwise $U$ could not be generated by $S$ by the above remark.

We now show that $\tilde{X}$ is an infinite sheeted cover of $X$, i.e. that the last puncture does not correspond to a compact boundary component of  $\tilde X$. Suppose that $\tilde X$ is a $q$-sheeted cover of $X$. As $\chi(X)=2-(n+1)=1-n$ it follows that $\chi(\tilde X)=q(1-n)=q-qn$. Clearly $\tilde X$ has at least $qn+1$ boundary  components as $n$ boundary components of $X$ must have $q$ lifts each in $\tilde X$. It follows that $$\chi(\tilde X)\le 2-(qn+1)=1-qn.$$ Thus $q=1$ and therefore $U=F_n$, a contradiction.

Thus $\tilde{X}$ is an infinite sheeted cover of $X$ whose interior is homeomorphic to a punctured sphere such that all but one punctures correspond to compact boundary components of $\tilde{X}$ and such that the conjugacy classes of peripheral subgroups of $U$ corresponding to these boundary are represented by  some $\langle g_ix_{j_i}g_i^{-1}\rangle$. This implies that there exists $T$ satisfying (1) and (2), indeed we take $T$ to be the tuple of elements corresponding to the compact boundary components $\tilde{X}$ . Item (3) is obvious.
\end{proof}


We now use Lemma~\ref{L2} to describe subgroups of $A:=\pi_1 CS(\beta)$  that are generated by finitely many meridians, i.e. conjugates of $x_1$. Here $\beta$ is an $n$-strand braid such that associated permutation  $\tau\in S_{n}$ of $\beta $ is an $n$-cycle. Recall that $A$ is generated by $\{x_1,\ldots ,x_n,t\}$, that $\langle \langle x_1\rangle\rangle=\langle x_1,\ldots ,x_n\rangle$ is free in $\{x_1,\ldots ,x_n\}$ and that in $A$ the element $x_1$ is conjugate to $x_i$ for $1\le i\le n$.  As in the previous section we denote the peripheral subgroups of $A$  by $P$ and~$C$.

\begin{corollary} {\label{C1}}
Let $S=\{ g_ix_1g_i^{-1} \vert 1\leq i\leq k \}$ with $k\geq 0$ and $g_i \in A$ for $1\le i\le k$. Suppose that $U:=\langle S\rangle \leq A$.

Then either  $U=\langle\langle x_1\rangle\rangle$  or  there exists 
$$T=\{ {g'}_lx_1{g'}_l^{-1} \vert 1\leq l\leq {m}\}$$  
with ${m}\leq k$ and ${g'}_l \in A$ for $1\le l\le m$    such that the following hold:
\begin{enumerate}
\item $U$ is freely generated by $T$.
\item For any $g\in A$ one of the following holds:
\begin{enumerate}
\item  $gPg^{-1}\cap U=\{1\}$.
\item  $gPg^{-1}\cap U=g\langle x_1 \rangle g^{-1}$ and $gx_1g^{-1}$ is in $U$ conjugate to ${g'}_lx_1{g'}_l^{-1}$ for  some $l\in\{1,\ldots,{m}\}$.
\end{enumerate}
 
\item For any $g\in A$  we have $gCg^{-1}\cap U=\{1\}$.
\end{enumerate}
\end{corollary}
\begin{proof}
For any $g=u\cdot t^r\in A$ we have $gx_1g^{-1}=wx_{i}w^{-1}$ 
where    $i=\tau^r(1)\in \{1,\ldots,n\}$ and $w\in F_n$.  Hence $S$ can be rewritten in the form  
$$S=\{w_1x_{j_1}w_1^{-1},\ldots, w_kx_{j_k}w_k^{-1}\}$$
with $w_i \in F_n$ and $j_i\in \{1,\ldots,n\}$ for $1\le i\le k$. By Lemma~\ref{L2} there exists 
$$ T=\{ {w'}_1x_{p_1}{w'}_1^{-1},\ldots, {w'}_{m}x_{p_{m}}{w'}_{m}^{-1}\} $$
with $m\leq k$ and ${w'}_l \in F_n$ and $p_l\in \{1,\ldots,n\}$ for $1\le l\le m$ such that $U$ is freely generated by $T$ and that the other conclusions of Lemma~\ref{L2} are satisfied. We will show that $T$ satisfies (1)-(3).

Now each $w'_lx_{p_l}{w'_l}^{-1}$ can be written as ${g'}_lx_1{g'}_{l}^{-1}$ for some $g'_l\in A$. Hence $T$ satisfies (1).

Note that for any $g\in A$ we have 
$$gPg^{-1}\cap U\le gPg^{-1}\cap F_n = gPg^{-1}\cap gF_ng^{-1}= g(P\cap F_n)g^{-1}=g\langle x_{1} \rangle g^{-1}.$$
Suppose that $gPg^{-1}\cap U\neq 1$, i.e.  $gx_{1}^zg^{-1}\in U$ for  some  integer $z\neq 0$. It follows from Lemma~\ref{L2}(3) and the fact that any element in a  free group is contained in a unique maximal cyclic subgroup that $gx_{1}g^{-1}\in U$  and that $gx_{1}g^{-1}$ is in $U$ conjugate to ${g'}_lx_{1}{g'}_l^{-1}$ for some $l\in \{1,\ldots,m\}$, thus we have shown that $T$ satisfies (2). 

(3) For any $g=w\cdot t^r\in A$ we have $$gCg^{-1}\cap U\le gCg^{-1}\cap F_n = gCg^{-1}\cap gF_ng^{-1}= g(C\cap F_n)g^{-1}=g\langle x_{n+1} \rangle g^{-1}.$$
If  $gx_{n+1}^zg^{-1}=wx_{n+1}^zw^{-1}\in U$ for some integer $z\neq 0$, then    Lemma~\ref{L2} implies that   $wx_{n+1}^zw^{-1}$ is in $U$   conjugate to  $w'_{l }x_{p_{l}}^z{w'}_{l }^{-1}$ for some $l$. But this is a contradiction since in $F_n$  the element    $x_i^z$ is not   conjugate to $x_j^z$ for $i\neq j\in \{1,\ldots,n+1\}$. 
\end{proof}

We now use Lemma~\ref{L2} to describe subgroups of 
$$B=\pi_1W_n=\langle x_1,\ldots,x_n,t | [x_i,t]=1\rangle$$
that are generated   by $t$ and conjugates of the $x_i$.    Here $P=\langle x_{n+1},t\rangle $ and $C_i=\langle x_i,t\rangle$ for $1\le i\le n$.

\begin{corollary}\label{C3} 
Let $S=\{ g_ix_{j_i}g_i^{-1} \vert 1\leq i\leq k \}$ with   $g_i \in B$  and $j_i\in\{1,\ldots,n\}$ for $1\le i\le k$. Let $U:=\langle S\rangle \times \langle t\rangle\leq B$.

Then either  $U=B$  or  there exists 
$$T=\{ {h}_lx_{p_l}{h}_l^{-1} \vert 1\leq l\leq {m}\}$$   
with ${m}\leq k$, ${h}_l \in F_n$ and $p_l\in\{1,\ldots ,n\}$ for $1\le l\le m$  such that the following hold:
\begin{enumerate}
\item $\{j_1,\ldots,j_k\}=\{p_1,\ldots,p_m\}$.
\item $\langle S\rangle$ is freely generated by $T$.
\item For any $g\in B$ and any $i\in \{1,\ldots,n\}$   one of the following holds: 
\begin{enumerate}
\item $gC_ig^{-1}\cap U=\langle t\rangle$.
\item $gC_ig^{-1}\cap U=gC_ig^{-1}=g\langle x_i \rangle g^{-1}\times \langle t\rangle$ and $gx_ig^{-1}$ is in $U$ conjugate to ${h}_lx_{p_l}{h}_l^{-1}$ for  some $l\in\{1,\ldots,{m}\}$. In particular $i=p_l$.
\end{enumerate}
\item For any $g\in B$  we have $gPg^{-1}\cap U=\langle t\rangle$.
\end{enumerate}
\end{corollary}
\begin{proof} 
For any $g=w\cdot t^z\in B$ and $1\le i\le n+1$ we have $$gx_ig^{-1}=wx_{i}w^{-1}.$$  Hence $S$ can be rewritten in the form  
$$S=\{w_1x_{j_1}w_1^{-1},\ldots, w_kx_{j_k}w_k^{-1}\}$$ with $w_i\in F_n$ for $1\le i\le k$
and so $\langle S\rangle\leq F_n=\langle x_1,\ldots,x_n\rangle$.   If $\langle S\rangle=F_n$ then $U=B$; thus we may assume that $\langle S\rangle$ is a proper subgroup of $F_n$.  By Lemma \ref{L2} there exists 
$$ T=\{ h_1x_{p_1}h_1^{-1},\ldots, h_{m}x_{p_{m}}h_{m}^{-1}\} $$
with $m\le k$ and $h_l \in F_n$ and $p_l\in \{1,\ldots,n\}$ for $1\le l\le m$ such that $U$ is freely generated by $T$ and that (1) and (2) are satisfied. We will show that $T$ satisfies (3)-(4).

(3) Let $g\in B$ and $i\in\{1,\ldots ,n\}$. Clearly  we have 
$$\langle t\rangle\subseteq gC_ig^{-1}\cap U.$$ 
If $\langle t\rangle= gC_ig^{-1}\cap U$ there is nothing to show. Thus we may assume that   $\langle t\rangle \varsubsetneq gC_ig^{-1}\cap U$.

It follows that $gx_i^zg^{-1}\in U$ for some $z\neq 0$. This clearly implies that  $gx_i^zg^{-1}\in \langle S\rangle$. It follows from Lemma~\ref{L2}(3) that $gx_{i}g^{-1}\in \langle S\rangle$  and that $gx_{i}g^{-1}$ is in $\langle S\rangle$ conjugate to ${h}_lx_{p_l}h_l^{-1}$ for some $l\in \{1,\ldots,m\}$, thus we have shown that $T$ satisfies~(3).

(4) For any $g=w\cdot t^z\in B$ we have 
$$\langle t\rangle \subseteq gPg^{-1}\cap U.$$

If $\langle t\rangle \varsubsetneq gPg^{-1}\cap U$ then  $gx_{n+1}^zg^{-1}=wx_{n+1}^zw^{-1}\in U$ for some integer $z\neq 0$. It follows that \color{black} $gx_{n+1}^zg^{-1}\in \langle S\rangle\leq U$.    Lemma~\ref{L2} implies that   $wx_{n+1}^zw^{-1}$ is in $\langle S\rangle$   conjugate to  $h_{l }x_{p_{l}}^z{h}_{l }^{-1}$ for some $l$. But this is impossible as in $F_n$  the element    $x_i^z$ is not   conjugate to $x_j^z$ for $i\neq j\in \{1,\ldots,n+1\}$. 
\end{proof}


\section{Proof of the main theorem}

 In this section we give the proof Theorem~\ref{main} or equivalently of Theorem~\ref{main2}. We will show that for any knot $\K$ from $\mathcal K$ the meridional rank $w(\K)$ is bounded from below  by the bridge number that is  given by Lemma~\ref{bridgenumber}.

\smallskip The general idea of the proof is similar to the proof of Grushko's theorem. Clearly there is an epimorphism from the free group of rank $w(\K)$ to $G(\K)$ that maps any basis element (of some fixed basis) to a conjugate of the meridian. This epimorphism can be a realized by a morphism of graphs of groups, see construction of $\mathcal B_0$ below.  Such a morphism can be written as a product of folds and the main difficulty of a proof is to define a complexity that does not increase in the folding sequence such that comparing the complexities of the initial and the terminal graph of groups yields the claim of the theorem.

\smallskip Morphisms of graphs of groups were introduced by Bass \cite{B}, we will use the related notion of $\mathbb A$-graphs as presented in \cite{RW} which in turn a slight modification of the language developed in \cite{KMW}, in fact we assume complete familiarity with the first chapted of \cite{RW} which in particular contains a detailed description of folds as introduced by Bestvina and Feighn \cite{BF} in the language of $\mathbb A$-graphs.

\smallskip The proof is by contradiction, thus  we assume that the knot group  $G(\K)$ of $\K$ is generated by $l:=w(\K)<b(\K)$ meridians, namely  by $g_1mg_1^{-1},\ldots,g_lmg_l^{-1} $ where $g_i\in G(\K)$ for $i=1,\ldots, l$.    Since  $G(\K)$ splits as $\pi_1(\A,v_0)$ where $\A$ is the tree of groups  described in Section~\ref{DSF},  we see that    for $1\leq i \leq l$ the element $g_i$ can be written as $g_i=[\gamma_i]$ where $\gamma_i$ is an $\A$-path from $v_0$ to $v_0$ of the form 
$$\gamma_i= a_{i,0},e_{i,1},a_{i,1},e_{i,2},\ldots,a_{iq_i-1},e_{i,q_i},a_{i,q_i}$$
for some $q_i\geq 1$. Observe that we do not require $\gamma_i$ to be reduced, otherwise we could possibly not choose $\gamma_i$ such that $q_i\ge 1$. 
  
\medskip We now define the $\mathbb{A}$-graph $\B_0$ as follows:
\begin{enumerate}
\item The underlying graph $B_0$ is a finite tree given by:
\subitem(a)  $EB_0 :=\lbrace f_{i,j}^{\varepsilon} \ \vert \ 1\leq i\leq l, 1\leq j\leq q_i,  \varepsilon\in \{-1,+1\}  \rbrace$.
\subitem(b)  $VB_0:= \{u_0\}\cup  \lbrace u_{i,j}  \ \vert \  1\leq i\leq l, 1\leq j\leq  q_i  \rbrace$.
\subitem(c) For $1\le i\le l$ the initial vertex of $f_{i,j}$ is
$$
\alpha(f_{i,j})=\left\{
\begin{array}{lcl}
u_0 & \text{if} &  j=1.  \\
u_{i,j-1} & \text{if}  & j>1. \\
 \end{array}
\right.
$$
while the  terminal vertex of $f_{i,j}$ is  $\omega(f_{i,j})=u_{i,j}$ for  $1\le j\le q_i$. 
\item  The graph morphism  $[\cdot]:B_0\rightarrow A$ is given by $[f_{i,j}^{\varepsilon}]=e_{i,j}^{\varepsilon}$.
\item For each $u\in VB_0\cup EB_0$ the associated group is
$$
B_x=\left\{
\begin{array}{lcl}
\langle m_{v_0}\rangle& \text{if} &  x=u_{i,q_i} \text { for some } 1\leq i\leq l.  \\
1 &   & \text{ otherwise. }  \\
 \end{array}
\right.
$$

\item For $1\leq i\leq l$,   $(f_{i,j})_{\alpha} = a_{i,j-1} $ for all $1\leq j\leq q_i$ while 
$$
(f_{i,j})_{\omega}=\left\{
\begin{array}{lcl}
1 & \text{if} &  j<q_i.  \\
a_{i,q_i} & \text{if}   & j=q_i. \\
 \end{array}
\right.
$$
\end{enumerate} 
 
Observe that the   fundamental group of the associated graph of groups $\mathbb{B}_0$ of $\B_0$  is  freely  generated by the elements 
$$y_i:=[1,f_{i,1},1,\ldots,f_{i,q_i},m_{v_0},f_{i,q_i}^{-1},\ldots,1,f_{i,1}^{-1},1]$$
for $1\leq i\leq l$.  Additionally, the induced   homomorphism $\phi:\pi_1(\mathbb{B}_0,u_0)\rightarrow \pi(\A,v_0)$ is surjective  as   $g_img_i^{-1}=\phi(y_i)$  by our construction of $\mathcal{B}_0$. 

\medskip Before we continue with  the proof we need some terminology.  Let $\B$ be an arbitrary $\A$-graph whose  underlying graph $B$ is a finite tree. Throughout the proof we assume that $B$  has the orientation $EB_+:=\{f\in EB \ | \ [f]\in E(A,v_0)\}$.

\smallskip We say that a vertex $u\in VB$ is \textit{isolated} if $B_f=1$ for any $f\in St_{+}(u,B)$.

\smallskip We further say that a vertex $u\in VB$ is \textit{full} if there exists a sub-tree $B^{\prime}$ of  $B(u)$ such that the following hold:
\begin{enumerate} 
\item $u\in VB'$.
\item  $[\cdot] :B \rightarrow A$ maps $B'$ isomorphically   onto $A([u])$.
\item $B_{x}=A_{[x]}$ for all $x\in VB'\cup EB'$.
\end{enumerate}

The following lemma  follows immediately from the definition of fullness: 
\begin{lemma}{\label{FV}}
Let $\B$ be an $\A$-graph whose underlying graph is a finite tree and $u\in VB$. Then $u$ is full if and only if $B_u=A_{[u]}$ and  there exists   $S_{+}\subseteq St_{+}(u,B)$ such that the following hold:
\begin{enumerate}
\item $[\cdot]|_{S_{+}} :S_{+}\rightarrow St_{+}([u],A)$ is a bijection.
\item $B_{e}=A_{[e]}$ for all $e\in S_{+}$.
\item $\omega(e)$ is full for all $e\in S_{+}$.
\end{enumerate}    
\end{lemma}

\begin{lemma}{\label{LFull}}
Let $\B$ be an $\A$-graph whose underlying graph is a finite tree. Assume that  $\B_1$ is obtained  from $\B$ by a fold. Then the image of any full vertex in $\B$ is full in $\B_1$.
\end{lemma}
\begin{proof}
First note that  if $\B'$ is  a folded $\A$-subgraph of $\B$, then it is easy to see that  $\B'$ is isomorphic to its image in   $\B_1$.  Assume that $u\in VB$ is a full vertex of $\B$. By definition there exists a sub-tree $B'\subseteq A(u)$ such that:
\begin{enumerate}
\item $u\in B'$.
\item $[\cdot]|_{B'}:B'\rightarrow A([u])$ is an isomorphism.
\item $B_x=A_{[x]}$ for all $x\in VB'\cup EB'$. 
\end{enumerate}
The $\A$-subgraph $\B'$ of $\B$  having $B'$ as its  underlying graph   is trivially folded. Thus  $\B'$ is isomorphic to its image in  $\B_1$ and so the image of $u$ in $\B_1$ is full. \end{proof}


\begin{definition}{\label{DT}} We call an $\A$-graph $\B$ with associated graph of groups $\mathbb{B}$  \textit{tame} if the graph $B$  underlying   $\B$ is a finite tree and  the following conditions hold: 
\begin{enumerate}
\item[(1)] For each $f\in EB_{+}$ with $e:=[f]$ one of the following holds:
\begin{enumerate}
\item $B_f=1$.
\item $B_f=\langle m_{e }\rangle$.
\item $B_f=A_{e}=\langle m_{e},l_{e}\rangle$ and $\omega(f) \in VB$ is full.  
\end{enumerate}

\item[(2)] For every vertex $u\in VB$ with  $v:=[u]\in V_0$  one of the following holds:
\begin{enumerate}
\item   $B_u$ is generated by $r_u<b(\K_v)$ meridians of $\K_v$.
\item $B_u=A_v=G(\K_v)$.
\end{enumerate}

\item[(3)] For every vertex $u\in VB$ with $v:=[u]\in V_1$ one of the following holds:
\begin{enumerate}
\item   $B_u$  is freely generated by finitely many conjugates of $m_v$. 
\item  $B_u=A_v$ and $u$ is full.
\end{enumerate}

\item[(4)] For every vertex $u\in VB$ with $v:=[u]\in V_2$ one of the following holds:
\begin{enumerate}
\item  $B_u=1$.
\item  There exists  $S_u\subseteq \{f\in St_{+}(u,B) \ | \ B_f=A_{[f]}\}$ such that  $B_u=F_u\times \langle m_v\rangle$, where $F_u$ is freely generated by 
$$\{ g_{f} l_{[f]} g_{f}^{-1} \ | \ f\in S_u\}$$
where $g_{f}\in A_v$ and  $f_{\alpha}l_{[f]}f_{\alpha}^{-1}$ is in $B_u$ conjugated to $g_{f}l_{[f]}g_{f}^{-1}$ for all $f\in S_u$. Moreover, if  $B_u=A_v$  then $u$ is full. 

Put  $[S_u]:=\{[f] \ | \ f\in S_u\}\subseteq St_{+}(v,A)$.
\end{enumerate}
\end{enumerate}
\end{definition}


Next we will define the complexity of a tame $\A$-graph $\B$. First we need to introduce the following  notion.  We define the \emph{positive height} of a vertex $v\in VA$ as 
$$
h_{+}(v):=h(v)n(v) $$  
where $h(v)$ denotes the height of a vertex in $A$ defined in Section 4.   Note that for any edge $e\in E(A,v_0)$ we have $h_+(\alpha(e))=h(\omega(e))$. In particular we have $h_+(v)=h(v)$ if $v\in V_2$ and  $h_+(v)=h(v)\cdot n_v$ if $v\in V_1$.

\begin{definition}{\label{Comp}}
Let $\B$ be a tame $\A$-graph.

We define \emph{the $c_1$-complexity} of $\B$ as 
$$ c_1(\B):= \sum_{\substack{u\in VB  \\   u \text{ isol.} }}h([u])\cdot w(B_u) -\sum_{\substack{ u\in VB\\ u \text{ isn't isol.}}} h_{+}([u])\cdot (val_{+}^{1}(u ,\B)-1)$$
where $val_+^1(u,\B):=|\{f\in St_+(u,B) \ | \ B_f\neq 1 \}|$.

We further define the $c_2$-complexity of $\B$ as
$$c_2(\B):=2|EB|-|E_2B|-\dfrac{1}{2}|E_1B|$$  
where  $E_iB\subset EB$ denotes the set of edges whose edge group is isomorphic to $\mathbb{Z}^i$ for $i=1,2$.

Lastly, the $c$-complexity of $\B$ is defines as
$$ c(\B)=(c_1(\B),c_2(\B))\in \mathbb{N}\times \mathbb{N} .$$    

\end{definition}

As   we want to compare complexities we endow the set $\mathbb{N}\times \mathbb{N}$ with the lexicographic order, i.e. we write  $(n_1,m_1)<(n_2,m_2)$ if one of the following occurs:
\begin{enumerate}
\item $n_1< n_2$.
\item $n_1=n_2$ and $m_1< m_2$.
\end{enumerate}

Observe that the $\mathbb A$-graph $\B_0$ defined above is   tame   as $B_x=\{1\}$ for all $x\in VB_0\cup EB_0 \setminus \{u_{i,q_i}  \vert  1\leq i \leq l\}$ and $B_{i,q_i}\leq A_{v_0}$ is infinite cyclic generated be the meridian $ m  $ of $\K$. Furthermore, $c_1(\B_0)=l<b(\K)$. Thus our assumption (which will yield a contradiction) implies that there exists a tame $\A$-graph $\B_0$ with $c_1$-complexity strictly smaller than the bridge number of $\K=\K_{\mathcal{A}}$ such that the induced homomorphism 
$$\phi: \pi_1(\mathbb{B}_0,u_0)\rightarrow \pi_1(\A,v_0)$$ 
is an epimorphism.

\smallskip Let now $\B$ be a tame $\A$-graph such that there exists a vertex $u_0\in VB$ of type $v_0$ such that $\phi:\pi_1(\mathbb{B},u_0)\rightarrow \pi_1(\A ,v_0)$ is surjective and that among all such $\B$ the $c$-complexity is minimal. Note that $c_1(\B)\leq l<b(\K)$ as the $\A$-graph  $\B_0$  constructed above is tame and the map $\phi$ is sujective. We will prove the theorem by deriving a contradiction to this minimality assumption.  


 \begin{lemma} $\mathcal B$ is not folded.
 \end{lemma}
 
 \begin{proof}
Assume that $\mathcal B$ is folded. As the map $$\phi:\pi_1(\mathbb{B},u_0)\rightarrow \pi_1(\A,v_0)$$  is surjective it follows that $\mathbb{B}$ is isomorphic to $\mathbb A$, i.e. the graph morphism is bijective and that all vertex and edge groups are mapped bijectively. Observe that this implies the following:
\begin{enumerate}
\item The morphism $[\cdot]$ maps the isolated vertices of $VB$ bijectively to $V_0\subset VA$. It follows in particular that for any isolated vertex $u\in VB$ we have $w(B_u)=w(A_{[u]})=w(G(\K_{[u]}))=b(\K_{[u]})$ as $\K_{[u]}$ is meridionally tame by assumption.
\item As all edge groups are non-trivial we have $val_+^1(u,\B)=val_{+}(u, B)=val_{+}([u],A)$ for all $u\in VB$.
\item If $u\in VB$ is not isolated, then either $[u]\in V_1$ and therefore $val_+^1(u,\mathcal B)-1=0$ or $[u]\in V_2$ and therefore $h_+([u])=h([u])$.
\end{enumerate}  Using Lemma~\ref{bridgenumber} this implies that  
$$b(\K)=\bigg[\sum_{v\in V_{0}}h(v)\cdot b({\K}_v)\bigg ]-\bigg[\sum_{v\in V_{2}}h(v)\cdot (val_{+}(v,A)-1)\bigg]=$$ 
$$=\sum_{\substack{u\in VB  \\   u \text{ isol.} }}h([u])\cdot w(B_u) -\sum_{\substack{ u\in VB\\ u \text{ isn't isol.}}} h_{+}([u])\cdot (val_{+}^{1}(u ,\B)-1)=c_1(\mathcal B)$$ contradicting the assumption that $c_1(\mathcal B)<b(\K)$. 

Thus $ \mathcal B$ is not folded.
\end{proof}


As the   $\A$-graph $\B$ is not folded, a fold can be applied to $\mathcal B$.  As the graph $B$ underlying $\B$ is a tree it follows that only folds of type IA and IIA can occur. We will only apply folds of type IIA if no fold of type IA can be applied. As any fold is a composition of finitely many auxiliary moves (that clearly preserve tameness and  do not change the complexity) and an elementary move we can assume that one of the following holds:
\begin{enumerate}
\item An elementary move of type IA can be applied to $\B$.
\item No fold of type IA can be applied to $\B$ but an elementary move of type IIA can be applied to $\B$. 
\end{enumerate}

\smallskip In both cases we will derive the desired contradiction to the minimality of  the complexity  $\B$ by producing a tame $\A$-graph $\mathbb B''$  that is $\pi_1$-surjective  and such that  $c(\B'')<c(\B)$.

\smallskip The following lemma  will be useful when considering folds. It implies that certain type IIA folds  are only possible if also a fold of type IA is possible. Because of our above choice we will therefore not need to consider such folds  of type IIA.


\begin{lemma}\label{L7}
Let $\B$ be a tame $\A$-graph, $u\in VB$ and $v:=[u]$. Suppose that one of the following holds:
\begin{enumerate}
\item  $v\in V_2$ and there exists  $f\in St_+(u,B)$   labeled $(a,e,b)$ such that $B_f\neq A_e$  and  $a\alpha_e(l_e)a^{-1}=al_ea^{-1}\in B_u$.
\item $v\in V_1$, $\langle\langle m_v\rangle\rangle \leq B_u$ and there exist distinct edges $f_1,f_2\in St_{+}(u,B)$. 
\end{enumerate}
 Then we can apply a fold of type IA to $\B$.
\end{lemma}
\begin{proof}
(1) Note first that an element $$g\in A_v=F_v\times \langle m_v\rangle=F(\{l_e \ | \ e\in St_{+}(v,A)\})\times\langle m_v\rangle$$ commutes with $l_{e}$ if and only if $g=m_v^{z_1} \cdot l_{e}^{z_2}=\alpha_e(m_e^{z_1}\cdot l_e^{z_2}) $ for  $z_1,z_2\in\mathbb Z$. This follows immediately from the fact that any maximal cyclic subgroup of $F_v$ and therefore also  $\langle l_e\rangle$ is self-normalizing in $F_v$. 

It follows from Corollary~\ref{C3}(3.b) and the tameness of $\B$ that there exists an edge $f_0\in St_{+}(u,B)$ labeled $(a_0,e,b_0)$ such that $B_{f_0}=A_e=\langle m_{e},l_{e}\rangle$ and  
$$al_{e}a^{-1}=ga_0 l_{e}a_{0}^{-1}g^{-1}$$ 
for some $g\in B_u$. Thus, $a_0^{-1}g^{-1}a$ commutes with $l_{e}$ which implies that 
$$a_0^{-1}g^{-1}a=  m_v^{z_1}\cdot l_{e}^{z_2}=\alpha_e(m_{e}^{z_1}\cdot l_{e}^{z_2})$$ 
for $z_1,z_2\in \mathbb{Z}$. Hence $a=ga_0\alpha_e(m_{e}^{z_1}l_{e}^{z_2})$. Since we  further have $f_0 \neq f$ and $[f_0]=[f]=e$ it follows that  $\B$ is not folded because condition (F1) is not satisfied, see p.615 from \cite{RW}.  Thus we can apply a fold of type I or  type III to $\B$. Since the underlying graph $B$ of $\B$ is a tree it follows that we can apply a fold of type IA to~$\B$.

(2) Since $v\in V_1$ it follows that $St_{+}(v,A)=\{e\}$ for some $e\in EA$ and so $f_1$ and $f_2$ are of the same type. Let  $(f_i)_{\alpha}=w_i\cdot t^{z_i}$  where $w_i\in \langle\langle m_v\rangle\rangle $ and $z_i\in \mathbb{Z}$ for $i\in \{1,2\}$. We can write $(f_2)_{\alpha}\in A_v=F_{n_v}\rtimes \langle t\rangle$ as 
$$(f_2)_{\alpha}= w_2w_1^{-1} \cdot (f_1)_{\alpha}\cdot t^{z_2-z_1}.$$
Note now that  $ w_2w_1^{-1} \in \langle\langle m_v \rangle\rangle \leq B_u$ and  $t^{z_2-z_1}=\alpha_e(l_e^{z_2-z_1})$ since $A_e=\langle m_e,l_e\rangle$ and $\alpha_e(l_e)=t$. Thus $\B$ is not folded because condition (F1) is not satisfied which implies in our context    we can apply a fold of type IA to $\B$.   
\end{proof}


The following Lemma implies that if we apply an elementary move of type IIA in the direction of an oriented edge, we only ever add a meridian to the edge group. 
\begin{lemma}{\label{L8}}
Let $\B$ be a tame $\A$-graph and $f\in EB_{+}$ labeled $(a,e,b)$ with  $x:=\alpha(f)$. Assume that no  fold of type IA is applicable to $\B$(i.e. condition (F1) is satisfied) and that  
$$B_f\varsubsetneq \alpha_e^{-1}(a^{-1}B_{x}a).$$ 

Then $\alpha_e^{-1}(a^{-1}B_xa)=\langle m_{e}\rangle$. In particular,  $B_f=1$.
\end{lemma}
\begin{proof}
We first assume that $v:=\alpha(e)\in V_1$. By tameness of $\B$ we know that $B_x=A_v$ or $B_x\leq \langle \langle m_v\rangle\rangle$ is freely generated by conjugates of $m_v$. 

If $B_x=A_v$,  then the tameness of $\B$ implies that $x$ is full. Hence   there exists an edge $f_0\in St_{+}(x,B)$ such that  $B_{f_0}=A_e$. Since no folds of type IA are applicable to $\B$, it follows from Lemma~\ref{L7} that $f=f_0$ which implies that 
$$\alpha_e^{-1}(a^{-1}B_xa)=\alpha_e^{-1}(A_v)=A_e=B_f$$ 
contradicting the fact that $B_f\varsubsetneq \alpha_e^{-1}(a^{-1}B_xa)$. 

Thus $B_x\leq \langle\langle m_v\rangle\rangle$ is freely generated by conjugates of $m_v$. As by hypothesis $B_f\varsubsetneq \alpha_e^{-1}(a^{-1}B_xa)$ it implies that $\alpha_e^{-1}(a^{-1}B_xa)$ is non-trivial. It is a consequence of  Corollary~\ref{C1}(3)  that in this case  $B_x=\langle \langle m_v\rangle\rangle$ and so $\alpha_e^{-1}(a^{-1}B_xa)=\langle m_{e}\rangle$. 

\smallskip Assume now that $v\in V_2$. We know from Corollary~\ref{C3} that one of the following holds true:
\begin{enumerate}
\item  $\alpha_e^{-1}(a^{-1}B_xa)=1$
\item  $\alpha_e^{-1}(a^{-1}B_xa)=\langle m_{e}\rangle$
\item  $\alpha_e^{-1}(a^{-1}B_xa)=A_e=\langle m_{e},l_{e}\rangle$
\end{enumerate}
If $\alpha_e^{-1}(a^{-1}B_xa)=A_e $, then $al_{e}a^{-1}\in F_x$. As by assumption $B_f\varsubsetneq A_e $ we conclude from  Lemma~\ref{L7}(1)  that a fold of type IA can be applied on $\B$, a contradiction. Thus (1) or (2) occurs. As   $B_f \varsubsetneq  \alpha_e^{-1}(a^{-1}B_x a)$ it implies that $\alpha_e^{-1}(a^{-1}B_x a)$ is non-trivial and so (2) must occur.
\end{proof}


We will now deal with the two cases mentioned above. In both cases we let $\B'$ be the $\mathbb A$-graph obtained from $\B$ by the fold of type IA or IIA, respectively. After possibly applying elementary moves first we can assume that the folds are elementary.

\medskip \textbf{Case 1: An elementary move of type IA can be applied to $\B$:} Let $\B'$ be the $\mathbb A$-graph obtained from $\B$ by this elementary move. Thus there exist distinct edges $f_1,f_2\in EB$ with  $\alpha(f_i)= u\in VB$   and   $l(f_i)=(a,e,b)$ for $i\in \{1,2\}$ such that $\B{'}$ is obtained from $\mathcal B$ by identifying $f_1$ and $f_2$ into a single edge $f$. Put  $x:=\omega(f_1)$,  $y:=\omega(f_2)$,  $v:=\alpha(e)$ and $w:=\omega(e)$ and $z:=\omega(f)$. In particular we have ${B_f'}=\langle B_{f_1},B_{f_2}\rangle$ and ${B_z'}=\langle B_x,B_y\rangle$.  We further  denote by ${B'}$  the underlying graph of $\B'$, see Figure~\ref{Fig3}. 
\begin{figure}[h] 
\begin{center}
\includegraphics[scale=1]{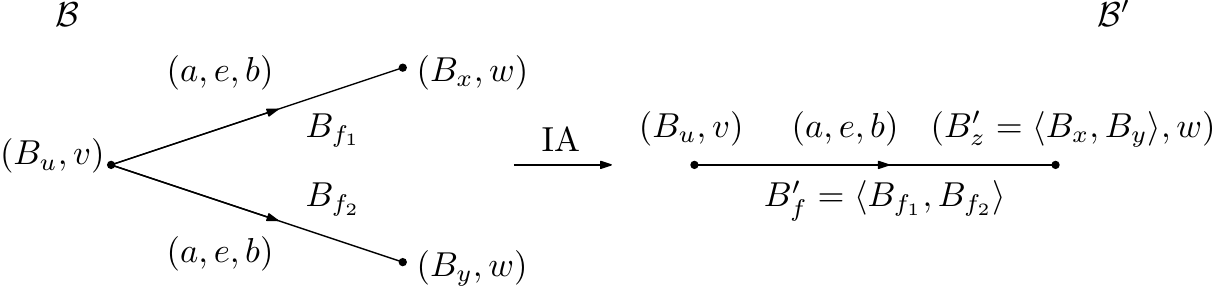}
\end{center}
\caption{An elementary  move of type IA.}\label{Fig3}
\end{figure}

Let now $\B''$ be the $\A$-graph obtained from $\B'$ by replacing the vertex group ${B_z'}$ by the group  ${B_z''}$ defined as follows:
$$
{B_z''}:=\left\{
\begin{array}{lcl}
{B_z'}=\langle B_x,B_y\rangle & \text{if} & w\in V_1\cup V_2 \text{ or } w\in V_0 \text{ and } w({B_x})+w(B_y)<b(\K_w).  \\
A_w=G(\K_w) & \text{if}  &  w\in V_0 \text{ and } w({B_x})+w(B_y)\geq b(\K_w). \\
 \end{array}
\right.
$$

From Proposition~8 of \cite{RW} it follows that $\B'$ is $\pi_1$-surjective which implies that $\B''$ is $\pi_1$-surjective since $\B''$ is defined from  $\B'$ by possibly enlarging the  vertex group at $z\in VB'$.  Note that the underlying graph $B''$ of $\B''$ is equal to  $B'$. Moreover, ${B_x''}={B_x}$ for all $x\in VB''\setminus \{z\}\cup EB''\setminus\{f,f^{-1}\}$ while  ${B_f''}={B_f'}=\langle B_{f_1},B_{f_2}\rangle$.

\begin{lemma}
The $\A$-graph $\B''$ is tame.   
\end{lemma}
\begin{proof}
We first show that edge groups are as in Definition~\ref{DT}. From the  tameness of $\B$ we conclude that one of the following occurs:
\begin{enumerate}
\item ${B_f''}=1$.
\item ${B_f''}=\langle m_e\rangle$.
\item ${B_f''}=A_e$. This occurs if and only if $B_{f_i}=A_e$  for some $i\in\{1,2\}$. In this case $\omega(f_i)$ is full by tameness and therefore $z$  is also  full by  Lemma~\ref{LFull}.
\end{enumerate}
Since the remaining edge groups do not change we conclude that edge groups are as in Definition~\ref{DT}.

We can easily see that  $\B''$ is tame if $w\in V_0$. Indeed, by definition ${B_z''}=G(\K_w)$ or ${B_z''}= \langle B_x,B_y\rangle$ is generated by at most  $w({B_x})+w(B_y)<b(\K_w)$ meridians of $\K_w$. Hence condition (2) of Definition~\ref{DT} is fulfilled. 
 
We next show that $\B''$ is tame if $w\in V_1$. Note first that if $B_q=A_w$ for some $q\in \{x,y\}$, then obviously ${B_z''}={B_z'}=A_w$. Tameness of $\B$ implies that $q\in VB$ is full and so $z$ is full by Lemma~\ref{LFull}.  

Thus we may assume that both  $B_x$ and $B_y$ are freely generated by conjugates of $m_w$. It follows from Corollary~\ref{C1} that  ${B_z''}=\langle B_x,B_y\rangle$ is freely generated by at most $w(B_x)+w(B_y)$ conjugates of $m_w$.   Since any subgroup of $A_w$ that is generated by conjugates of $m_w$ is contained in  $\langle\langle m_w\rangle\rangle$,  which in turn is a proper subgroup of $A_w$,  we conclude that ${B_z''}=A_w$ if and only if $B_q=A_w$  for some $q\in \{x,y\}$. Therefore $\B''$ is tame.

Lastly,  suppose that $w\in V_2$.   If $B_x=1$ or $B_y=1$, say $B_x=1$, then clearly $B_z''=B_y$ and there is nothing to show. Thus we may assume that for $q\in \{x,y\}$ we have $B_q=F_q\times \langle m_w\rangle $ with 
$$ F_q=F(\{ g_{f'}l_{ [f'] }g_{f'}^{-1} \ | \ f'\in S_q \})$$
where $S_q\subseteq \{f'\in St_{+}(q,B) \ | \ B_{f'}=A_{[f']}\}$ and  $g_{f'}\in A_w$ for all $f'\in S_q$. Thus $B_z''=F_z''\times \langle m_w\rangle$ where 
$$F_z'':=\langle \{g_{f'}l_{ [f'] }g_{f'}^{-1}  \ | \  f'\in S_x\cup S_y \}\rangle \leq F_w. $$
From Corollary~\ref{C3} we conclude  that there exists  a subset 
 $S_z''\subseteq S_x\overset{.}{\cup} S_y $ 
such that 
$$F_z''=F(\{h_{f'}l_{ [f'] }h_{f'}^{-1} \ | \ f'\in S_z'' \})$$  
and   $h_{f'}l_{ [f'] }h_{f'}^{-1}$ is in $B_z$ conjugated to $g_{f'}l_{ [f'] }g_{f'}^{-1}$ for all $f'\in S_z''$.  From tameness we know that $(f')_{\alpha}l_{ [f'] }(f')_{\alpha}^{-1}$ is conjugated to $g_{f'}l_{ [f'] }g_{f'}^{-1}$ and so it is also conjugated to $h_{f'}l_{ [f'] }h_{f'}^{-1}$. 
 
We next show  that $S_z''$ is preserved by the fold, i.e. that $S_z''$  is mapped injectively into $St_{+}(z,B'')$. This is trivial if $f_1,f_2\in EB_{+}$ since  
$$St_{+}(z,B'')=St_{+}(x,B)\overset{.}{\cup} St_{+}(y,B).$$
If $f_1,f_2\in EB\setminus EB_{+}$, then it is clear that  
$$St_{+}(z,B'')=(St_{+}(x,B)\setminus\{{f_1}^{-1}\})\cup (St_{+}(y,B)\setminus \{{f_2}^{-1}\}) \cup \{f^{-1}\}. $$
As $F_z''$ is free on the set $\{h_{f'}l_{ [f'] }h_{f'}^{-1} \ | \ f'\in S_z''\}$  and $l(f_i^{-1})=(b^{-1},e^{-1},a^{-1})$ for $i=1,2$, we conclude that $S_z''$ contains at most one element of $\{f_1^{-1},f_2^{-1}\}$ which implies that $S_z''$ is mapped injectively into $St_{+}(z,B'')$.

Finally we need to show that $z$ is full if $B_z=A_w$. Observe that if  $F_z''=F_w$ then  $[S_z'']=St_{+}(w,A)$ as $F_w$ is freely generated by  $\{ l_{e} \hspace{0.5mm} |\hspace{0.5mm} e\in St_{+}(w,A)\}$. Thus there exists  
$$S_{+}''\subseteq S_{z}''\subseteq \{ f'\in St_{+}(z,B'') \ | \ B_{f'}=A_{[f']}\}$$ 
such that $[\cdot]:B''\rightarrow A$ maps $S_{+}''$  bijectively  onto $St_{+}(w,A)$.  From Lemma~\ref{LFull} we conclude that for all $f'\in S_{+}''$ the vertex  $\omega(f')$ is full in $\B''$ as it is full in $\B$  and hence $z$ is full in $\B''$ by  Lemma~\ref{FV}.
\end{proof}

\begin{lemma}
$c(\B'')<c(\B)$.
\end{lemma}

\begin{proof}

\textbf{Case i(a): } Assume that $f_1,f_2\in EB_{+}$ and $B_{f_1}=1$ or $B_{f_2}=1$.

If $x$ and $y$ are isolated in $\B$, then obviously   $z$   is isolated in $\B''$  and   we clearly have $w(B_z'')\leq w(B_z)+w(B_y)$.     Thus
$$c_1(\B'')=c_1(\B)-h(w)w(B_x)-h(w)w(B_y)+h(w)w(B_z)\leq c_1(\B).$$

If one of the vertices  $x$ or $y$, say $x$,  is not isolated in $\B$, then $z$ is not isolated in $\B''$. In this case we have $val_{+}^{1}(z,\B'')=val_{+}^{1}(x,\B)+val_{+}^{1}(y,\B)$ and consequently 
$$
c_1(\B'')=\left\{
\begin{array}{lcl}
c_1(\B)-h(w)w(B_y) & \text{if} & y \text{ is isolated in  } \B.  \\
c_1(\B)-h_{+}(w)  & \text{if}  & y \text{ isn't isolated in }   \B.
 \end{array}
\right.
$$
Since $w(B_y)\geq 0$ and $h_{+}(w)\geq  h(w)\geq 1$ we get  $c_1(\B'')\leq c_1(\B)$. 

As at least  one of the edges involved on the fold has trivial group we obtain 
$$c_2(\B'')=c_2(\B)-4.$$ 
Therefore $c(\B'')<c(\B)$.
 
\textbf{Case i(b): }  Assume that $f_1,f_2\in EB_{+}$ and  $B_{f_1}\neq 1 \neq B_{f_2}$. Thus  
$$val_{+}^{1}(u,\B'')=val_{+}^{1}(u,\B )-1.$$ 

If both $x$ and $y$ are isolated in $\B$, then clearly $z$ is isolated in $\B''$.  Moreover,  $w(B_z'')\leq w(B_x)+w(B_y)-1$. Indeed, for $w\in V_0$ the inequality  follows from the fact that  $\K_w$ is meridionally tame and  $B_{f_i}\neq 1$ for $i\in \{1,2\}$.  For $w\in V_1\cup V_2$ we know from the tameness of $\B$ that $B_q$ is generated by conjugates of $m_w$ since $q$ is isolated.  Thus  $B_{f_i}=\langle m_e\rangle$ since otherwise the vertex $\omega(f_i)$ is full.  Hence for $w\in V_1$ the inequality follows from Corollary~\ref{C1}(2.b) and for $w\in V_2$ this is trivial since   $B_x=B_y=\langle m_w\rangle$ and hence $B_z''=\langle m_w\rangle$. As $h_{+}(v)=h(w)$ we obtain
\begin{eqnarray}
c_1(\B'')&   =   & c_1(\B)+h(w)(w(B_z)-w(B_x)-w(B_y))+\nonumber\\
         &       & + h_{+}(v)(val_{+}^1(u,\B )-val_{+}^{1}(u,\B'')) \nonumber\\
         & \leq  & c_1(\B)-h(w)+h_{+}(v)\nonumber\\
         &   =   & c_1(\B).\nonumber 
\end{eqnarray}

Now suppose that $x$ is not isolated in $\B$, the case that   $y$ is not isolated in $\B$ is analogous. Hence $z$ is not isolated in $\B'' $. Note that 
$$val_{+}^{1}(z,\B'')=val_{+}^{1}(x,\B)+val_{+}^{1}(y,\B)$$ 
and  $h_{+}(v)=h(w)\leq h_{+}(w)$. Since $B_{f_2}$ is non-trivial it follows that $w(B_y)\geq 1$. Thus  if $y$ is isolated in $\B$ we obtain
$$ c_1(\B'')= c_1(\B)-h(w)w(B_y)+h_{+}(v)=c_1(\B)+h(w)(1-w(B_y))\leq c_1(\B).$$
If $y$ is not isolated in $\B$ then 
$$c_1(\B'')=c_1(\B)-h_{+}(w)+h_{+}(v) \leq c_1(\B).$$

Since  $c_2(\B'')\leq c_2(\B)-2$ we conclude that the $c$-complexity decreases.

\textbf{Case ii: } Suppose that $f_1, f_2\in EB\setminus EB_{+}$.

Initially  observe that  if $x$ and $y$ are isolated in $\B$ then   $z$ is isolated in $\B''$ and we clearly have $w(B_z'')\leq w(B_x)+w(B_y)$.  Hence
$$c_1(\B'')=c_1(\B)-h(w)(w(B_x)+w(B_y)-w(B_z))\leq c_1(\B).$$

If $x$ or $y$, say $x$, is not isolated  in $\B$ then $z$ is not isolated in $\B''$. Furthermore, we have
$$
val_{+}^{1}(z,\B'')=\left\{
\begin{array}{lcl}
val_{+}^{1}(x,\B)+val_{+}^{1}(y,\B) & \text{if} & B_{f_1}=1 \text{ or } B_{f_2}=1.  \\
val_{+}^{1}(x,\B)+val_{+}^{1}(y,\B)-1 & \text{if}  & B_{f_1}\neq 1 \text{ and } B_{f_2}\neq 1.\\
\end{array}
\right.
$$
If $y$ is isolated in $\B$, then a simple calculation shows that 
$$c_1(\B'')= c_1(\B)-h(w)w(B_y).$$
If $y$ is not isolated in $\B$ then we see that 
$$
c_1(\B'')=\left\{
\begin{array}{lcl} 
c_1(\B)-h_{+}(w) & \text{if}  &   B_{f_1}=1 \text{ or } B_{f_2}=1.\\
c_1(\B) & \text{ if} &  B_{f_1}\neq 1 \text{ and } B_{f_2}\neq 1.
 \end{array}
\right.
$$
Since $h_{+}(w)\geq h(w)\geq 1$ and $w(B_y)\geq 0$ it follows that $c_1(\B'')\leq c_1(\B)$. 
In any case the $c_2$-complexity decreases by at least two so that the $c$-complexity decreases. 
\end{proof}


\textbf{Case 2: No fold of type IA can be applied to $\B$ but an elementary move of type IIA can be applied to $\B$:}  Let $\B'$ be the $\mathbb A$-graph obtained from $\B$ by this elementary move. Hence  there is an edge $f\in EB $ with label $(a,e,b)$,  initial vertex  $x=\alpha(f)$ labeled   $(B_x,v)$ and terminal vertex  $y=\omega(f)$ labeled $(B_y,w)$ such that 
$$B_f\neq \alpha_{e}^{-1}(a^{-1}B_xa) .$$

We will distinguish two cases depending on whether $f\in EB_{+}$ or $f\in EB\setminus EB_{+}$.   

\textbf{Case i: } Assume that  $f\in EB_{+}$. 

It follows from Lemma~\ref{L8} that $B_f=1$ and $ \alpha_e^{-1}(a^{-1}B_xa)=\langle m_e\rangle$.  Denote by $\B'$ the $\A$-graph obtained from $\B$ by this fold, i.e. we replace   $B_f=1$ by $B_f':=\langle m_e\rangle$ and $B_y$ by ${B_y'}:= \langle B_y,b^{-1}\omega_e(m_e)b\rangle =\langle B_y,b^{-1}m_wb\rangle$. From Proposition~8 of \cite{RW}  it follows that $\B'$ is $\pi_1$-surjective. 

\begin{figure}[h!] 
\begin{center}
\includegraphics[scale=1]{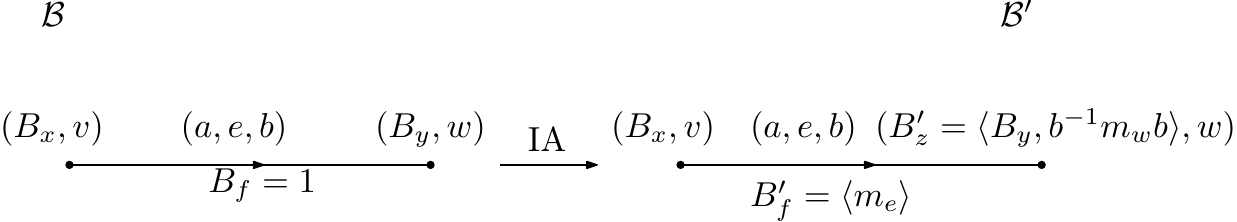}
\end{center}
\caption{An elementary  move of type IIA along an oriented edge.}\label{Fig4}
\end{figure}

Let now $\B''$ be the $\A$-graph obtained from $\B'$ by replacing the vertex group $B_y'$ by the group $B_y''$ defined as follows:
$$
{B_y''}:=\left\{
\begin{array}{lcl}
{B_y'}  & \text{if} & w\in V_1\cup V_2 \text{ or } w\in V_0 \text{ and } w({B }_y)+1<b(\K_w). \\
G(\K_w) & \text{if}  & w\in V_0 \text{ and }  w({B }_y) +1\geq b(\K_w). \\
 \end{array}
\right.
$$

Observe that the underlying graphs of $\B'$ and $\B''$ are both equal to the underlying graph $B$ of $\B$.  Moreover, $B''_u=B_u$ for all $u\in VB\setminus \{y\}\cup EB\setminus\{ f,f^{-1}\}$ while $B''_f=B'_f=\langle m_e\rangle$. 
By the same reasons as before we see that  $\B''$ is $\pi_1$-surjective. 
 
\begin{lemma}
The $\A$-graph $\B''$ is tame. 
\end{lemma}
\begin{proof}
Edge groups are as in Definition~\ref{DT} as ${B_f''}=\langle m_e\rangle$ and  the remaining edge groups do not change.

  It is trivial to see that the vertex group  $B_x$ is as in Definition~\ref{DT} since $B_x''=B_x$ and $B''_f=\langle m_e\rangle$.

It remains to check that the tameness condition is satisfied for $B''_y$. If $[y]=w\in V_0$ then this is immediate as ${B_y''}=G(\K_w)$ or ${B_y''}$ is generated by at most $w(B_y)+1<b(\K_w)$ meridians. 

We now show that $\B''$ is tame if  $w\in V_1$. Note first that  $B_y=A_w$ clearly implies  ${B_y''}= \langle B_y,b^{-1}m_wb\rangle =A_w$. From tameness it follows that $y$ is full and consequently  $y$ is full in $\B''$ by Lemma~\ref{LFull}.  Assume now  that $B_y$ is freely generated by $w(B_y)$ conjugates of $m_w$. It follows from Corollary~\ref{C1} that ${B_y''}=\langle B_y,b^{-1}m_wb\rangle$ is freely generated by at most $w(B_y)+1$ conjugates of $m_w$ and so is a proper subgroup of $A_w$. Thus we have showed that  $\B''$ is tame if $w\in V_1$. 

 Finally we show that $\B''$ is tame if $w\in V_2$. This case is trivial since $B_{f'}''=B_{f'}$ for all $f'\in St_{+}(y,B)$ and
 $$
{B_y''}:=\left\{
\begin{array}{lcl}
\langle m_w\rangle & \text{if} & B_y=1. \\
B_y & \text{if}  &  B_y\neq 1. \\
 \end{array}
\right.
$$
\end{proof}

\begin{lemma}
$c(\B'')<c(\B)$. 
\end{lemma}
\begin{proof}
In order to compute the complexity recall  that   $B_f=1$ is replaced by  ${B_f''}=\langle m_e\rangle $ which implies that $val_{+}^{1}(x,\B'')=val_{+}^{1}(x,\B)+1$. 

If $y$ is not isolated in $\B$ , then  $y$ is not isolated in $\B''$. A straightforward calculation shows that
$$
c_1(\B'')=\left\{
\begin{array}{lcl}
c_1(\B)-h(v)w(B_x)  & \text{if} &   x \text{ is isolated in  } \B.  \\
c_1(\B)-h_{+}(v)  & \text{if}  & x \text{ isn't isolated in }   \B.
 \end{array}
\right.
$$
As $w(B_x)\geq 1$ and $h_{+}(v)\geq h(v)\geq 1$ we obtain  $c_1(\B'')<c_1(\B)$. 

Assume now that   $y$ is isolated in $\B$ which implies that $y$ is  isolated in $\B''$. For the case in which $v\in V_1$ it was showed in  the   proof of Lemma~\ref{L8}  that $B_x=\langle\langle m_v\rangle\rangle$. Since no fold of type IA is applicable to $\B$, Lemma~\ref{L7} implies that $St_{+}(x,B)=\{f\}$. Hence the vertex $x$ is isolated in $\B$ and   its contribution to the $c_1$-complexity is $h([x])w(B_x)=h(v)n_{v}=h_+(v)=h(w)$. It is now  easy to see that
\begin{eqnarray}
c_1(\B'') & = & c_1(\B)-h(v)n_v-h(w)w(B_y)+h(w)w({B_y''})\nonumber\\
          & = & c_1(\B)-h(w)(1+w(B_y)+w(B''_y))\nonumber\\
          & \le & c_1(\B)\nonumber
\end{eqnarray}
as $w({B_y''})=w(\langle B_y,b^{-1}m_wb\rangle)\leq w(B_y)+1$.

Let us now consider that case in which $v\in V_2$.    If $x$ is isolated in $\B$, then its contribution to the $c_1$-complexity of $\B$ is $h(v)$ since in this case the tameness of $\B$ tells us that  $B_x=\langle m_v\rangle$. A simple calculation shows  that 
$$  c_1(\B'')=c_1(\B)-h(v)-h(w)w(B_y)+h(w)w(\langle B_y,b^{-1}m_wb\rangle). $$
Since $h(w)=h(v)$ we conclude that  $c_1(\B'')\leq c_1(\B)$. 

 If $x$ is not isolated in $\B$, then    $c_1(\B'')\leq c_1(\B)$ as $h_{+}(v)=h(w)$.

Finally, as $c_2(\B'')=c_2(\B)-1$it implies that $c(\B'')<c(\B)$.  
\end{proof}

\textbf{Case ii:} Assume now that $f\in EB\setminus EB_{+}$. Note that in this case  $w\notin V_0$.
The fact that the fold is possible,  combined with Corollaries \ref{C1} and \ref{C3},  imply that one of the following occurs:
\begin{enumerate}
\item $\alpha_e^{-1}(a^{-1}B_xa)=\langle m_{e}\rangle$.
\item $\alpha_e^{-1}(a^{-1}B_xa)=A_e=\langle m_e,l_e\rangle$. This occurs iff $B_x=A_v$.
\end{enumerate}

Denote by $\B'$ the $\A$-graph obtained from $\B$ by this fold, i.e. we replace $B_f$ by ${B_f'}$ defined as follows:
$$
{B_f'}:=\left\{
\begin{array}{lcl}
\langle m_e\rangle  & \text{if} & \alpha_e^{-1}(a^{-1}B_xa)=\langle m_e\rangle. \\
A_e & \text{if}  & \alpha_e^{-1}(a^{-1}B_xa)=A_e. \\
 \end{array}
\right.
$$
and the vertex group $B_y$ by ${B_y'}$ defined as 
$$
{B_y'}:=\left\{
\begin{array}{lcl}
 \langle B_y,b^{-1} \omega_e(m_e) b\rangle  & \text{if} & \alpha_e^{-1}(a^{-1}B_xa)=\langle m_e\rangle. \\
 \langle B_y,b^{-1}\omega_e(m_e)b,b^{-1}\omega_e(l_e)b\rangle & \text{if}  & \alpha_e^{-1}(a^{-1}B_xa)=A_e. \\
 \end{array}
\right.
$$
Once again, it follows from Proposition~8 of \cite{RW}  that   $\B'$ is $\pi_1$-surjective. 

Let $\B''$ be the $\A$-graph obtained from $\B'$ by replacing the vertex group ${B_y'}$ by ${B_y''}$ defined as follows:
$$
{B_y''}:=\left\{
\begin{array}{lcl}
{B_y'}  & \text{if} & w\in V_2. \\
\langle\langle m_w\rangle\rangle & \text{if}  & \alpha_e^{-1}(a^{-1}B_xa)=\langle m_e\rangle \text{ and } w\in V_1. \\
A_w & \text{if}  & \alpha_e^{-1}(a^{-1}B_xa)=A_e \text{ and } w\in V_1. \\
 \end{array}
\right.
$$
We remark that $\B''$ is $\pi_1$-surjective as it is defined from the $\pi_1$-surjective $\mathbb A$-graph $\mathcal B'$ by possibly enlarging the vertex group at $y$. 

\begin{lemma}
The $\A$-graph $\B''$ is tame. 
\end{lemma}
\begin{proof} Note that  ${B_f''}={B_f'}=\langle m_e\rangle$ or ${B_f''}={B_f'}=A_e=\langle m_e,l_e\rangle$. The latter case occurs only  when $B_x=A_{v}$. By tameness it follows that $x$ is full in $\B$ and hence  $x$ is full in $\B''$ by  Lemma~\ref{LFull}.

The $\A$-graph $\B''$ is obviously tame if $w\in V_1$ since by definition   ${B_y''}=\langle\langle m_w\rangle\rangle $ or $ {B_y''}=A_w$. The latter case occurs only  when $B_x=A_v$ and so ${B_f''}=A_e$. As $f^{-1}\in St_{+}(y,B)$ and $St_{+}(w,A)=\{e^{-1}\}$ we conclude from Lemma~\ref{FV} that $y$ is full in $\B''$.  

Finally we show that $\B''$ is tame if $w\in V_2$. It is trivial to see that $\B''$ is tame if  $\alpha_e^{-1}(a^{-1}B_xa)=\langle m_e\rangle$ since in this situation  we have 
$$
{B_y''}:=\left\{
\begin{array}{lcl}
\langle m_w\rangle & \text{if} & B_y=1. \\
B_y & \text{if}  & B_y\neq 1. \\
 \end{array}
\right.
$$
It is also trivial to see that $\B''$ is tame if $\alpha_e^{-1}(a^{-1}B_xa)=A_e$ and $B_y\leq \langle m_w\rangle$ since in this case ${B_y''}=\langle b^{-1}l_{e^{-1}}b\rangle\times \langle m_w\rangle \varsubsetneq A_w$. So 
assume that $\alpha_e^{-1}(a^{-1}B_xa)=A_e$ and $B_y=F_y\times \langle m_w\rangle$ where $F_y$ is freely generated by 
$$\{g_{f'}l_{ [f'] }g_{f'}^{-1} \ | \ f'\in S_y\}.$$
Thus  
$${B_y''}={F_y''}\times \langle m_w\rangle$$
where
$${F_y''}:=\langle \{b^{-1}l_{e^{-1}}b\}\cup \{g_{f'}l_{ [f'] }g_{f'}^{-1} \ | \ f'\in S_y\}\rangle.$$
Corollary~\ref{C3} implies that there exists a subset  $ {S_y''}\subseteq S_y\cup \{f^{-1}\}$ such that ${F_y''}$  is freely generated by  $\{h_{f'}l_{ [f'] }h_{f'}^{-1} \ | \  f'\in {S_y''}\}$ and $h_{f'}l_{[f']}h_{f'}^{-1}$ is in ${B_y''}$ is conjugated to $g_{f'}l_{[f']}g_{f'}^{-1}$  or to $b^{-1}l_{e^{-1}}b$. If ${B_y''}=A_w$ then ${F_y''}=F_w$ which implies that $[{S_y''}]=St_{+}(w,A)$. Combining Lemmas~\ref{FV} and \ref{LFull} and the tameness of $\B$  we conclude that $y$ is full in $\B''$. Therefore $\B''$ is tame if $w\in V_2$.  
\end{proof}

\begin{lemma}
$c(\B'')<c(\B)$.
\end{lemma}
\begin{proof}
Note that if $y$ is isolated in $\B$, then  we obtain 
$$c_1(\B'')=c_1(\B)-h(w)w(B_y)\leq c_1(\B) $$
since $y$ is not isolated in $\B''$ and $val_{+}^{1}(y,\B'')=1$. 

If $y$ is not isolated in $\B$,  then $val_{+}^{1}(y,\B'')\geq val_{+}^{1}(y,\B) $. Thus   
$$c_1(\B'')=c_1(\B)+h_{+}(w)(val_{+}^{1}(y,\B)-1)-h_{+}(w)(val_{+}^{1}(y,\B'')-1)\leq c_1(\B).$$

In both cases a simple calculation shows that $c_2(\B'')\leq c_2(\B)-1$ which implies that the $c$-complexity decreases. 
\end{proof}


\section{Meridional tameness of torus knots}

 In this section we show that torus knots are meridionally tame. This fact is implicit in \cite{RZ}.   Note that this   implies Corollary~\ref{CJSJ} since  any knot whose exterior is a graph manifold is equal to   $\K(\mathcal{A})$  for some labeled tree $\mathcal{A}$, where  $\K_v$ a torus knot for any $v\in V_0$. 

\begin{lemma}\label{lem:torus-knots} Torus knots are meridionally tame.
\end{lemma}
\begin{proof} Let $\K$ be a $(p,q)$-torus knot. Without loss of generality we may  assume that $  p>q\geq 2$. Suppose that   $k<b(\K)=min(p,q)=q$   and  $U=\langle  m_1,\ldots,m_k \rangle$ is a  meridional subgroup of $\K$.

It follows from Theorem~1.4 of \cite{RZ} that there exist $r\leq k$   meridians $m'_1,\ldots,m'_r\in G(\K)$   such    that  $U$ is freely generated by $\{m'_1,\ldots,m'_r\}$.   

We will show that for any $g\in G(\K)$ one of the following holds:
\begin{enumerate}
\item $gP(\K)g^{-1}\cap U=\{1\}$.
\item $gP(\K)g^{-1}\cap U=g\langle m\rangle g^{-1}$ and $gmg^{-1}$ is in $U$ conjugate to $m'_j$ for some $j\in \{1,\ldots,r\}$.
\end{enumerate}
Note that this claim  is implicit in the argument in \cite{RZ}, however, in order to avoid getting involved with the combinatorial details, we give an alternative argument. 

Observe that it suffices to show that any conjugate of a peripheral element that lies in $U$ is in $U$ conjugate to some element of $\langle m'_i\rangle$ for some $i\in \{1,\ldots,r\}$ since  the meridional subgroup $\langle m\rangle $ is premalnormal in $G(\K)$ with respect to the peripheral subgroup $P(\K)$, that is, $g\langle m\rangle g^{-1}\cap P(\K)\neq \{1\}$ implies $g\in P(\K)$, see  Lemma~3.1 of \cite{RW2}.
 
Let $\pi:G(\K) \to \mathbb Z_p*\mathbb Z_q$ be the map that quotients out the center. Note that we can think of $\mathbb Z_p*\mathbb Z_q$ as the fundamental group of the hyperbolic 2-orbifold $\mathcal O=S^2(\infty,p,q)$ of finite volume. Thus the boundary corresponds to a parabolic element. As $U$ is free and has therefore trivial center it follows that $\pi|_U$ is injective. Thus $\pi(U)$ is free in the parabolic elements $\pi(m'_1),\ldots ,\pi(m'_r)$.

Now consider the covering $\tilde{\mathcal O}$ of $\mathcal O$ corresponding to $\pi(U)$. As $\pi_1(\tilde{\mathcal O})$ is freely generated by $r$ parabolic elements it follows that $\tilde{\mathcal O}$ is an ($r+1$)-punctured sphere with at least $r$ parabolic boundary components. If the last boundary component corresponds to a hyperbolic element then the claim of the lemma holds. Thus we may assume that all boundary components correspond to parabolic elements, i.e., that $\tilde{\mathcal O}$ is a finite sheeted cover of $\mathcal O$ with \begin{equation}\label{equ1}\chi(\tilde{\mathcal O})=2-(r+1)=1-r> 1-\min(p,q)=1-q.\tag{$\ast$}\end{equation} 
Note that $\chi(\mathcal O)=1-(1-\frac{1}{p})-(1-\frac{1}{q})=-1+\frac{1}{p}+\frac{1}{q}$. As $\pi(U)$ is a torsion free subgroup of $\mathbb Z_p*\mathbb Z_q$ and as $p$ and $q$ are coprime it follows that $|\mathbb Z_p*\mathbb Z_q:\pi(U)|\ge p\cdot q$. Thus, as $\chi(\mathcal{O})<0$ we obtain
$$\chi(\tilde{\mathcal O})\le p\cdot q\cdot \chi(O)=p\cdot q\cdot \left(-1+\frac{1}{p}+\frac{1}{q}\right)=-p\cdot q+p+q.$$ 
As $2\leq q<p$ and therefore $p\ge 3$ it follows that 
$$\chi(\tilde{\mathcal O})\le -p\cdot q+p+q=p(1-q)+q\le 3(1-q)+q=3-2q\le1-q=1-\min(p,q),$$ 
a contradiction to (\ref{equ1}). Thus this case cannot occur. 
\end{proof}

\section{Meridional tameness of 3-bridge knots}\label{sec_3bridge}

\begin{proposition}\label{3bridge} A prime 3-bridge knot is meridionally tame.
\end{proposition}

\begin{proof}

From \cite[Corollary 1.6]{BJW} it follows that for a prime 3-bridge knot $K \subset S^3$  is either hyperbolic or a torus knot.

In the case of a torus knot the meridional tameness follows from Lemma~\ref{lem:torus-knots}. Thus we can assume that $K$ is hyperbolic.

\smallskip We need to check that any subgroup $U\le \pi_1(K)$ that is generated by at most two meridians is tame. It follows from \cite[Prop. 4.2]{BJW} that any such subgroup  is either cyclic  or the fundamental group of a 2-bridge knot summand of $K$ or free of rank $2$. In the first case  the tameness of $U$ is trivial and the second case cannot occur as $K$ is a prime 3-bridge knot. 

\smallskip Thus we are left with the case where the knot $K$ is hyperbolic and 
the meridional subgroup $U$ is a free group generated by two parabolic elements which are meridians. In \cite{MS} a Kleinian free group $U$ generated by two parabolics is shown to be geometrically finite. 
In the course of the proof it is proved that $U$ is of Schottky type and obtained from a handlebody of genus 2 by pinching at most 3 disjoint and non parallel simple curves to a point corresponding to cusps. Each curve generates a maximal parabolic subgroup of $U$ and there are at most 3 conjugacy classes of such subroups.

\smallskip If there are only two conjugacy classes of maximal parabolic subgroups, they correspond to the two meridian generators and so the group $U$ is tame. 
This is the 4-times punctured sphere case in \cite[case 4]{MS}  .

\smallskip The group $U$ may also have 3 conjugacy classes of maximal parabolic subgroups, corresponding to two 3-times punctured spheres on the boundary of the core of the associated Kleinian 
manifold, see  \cite[case 5]{MS}. They give $\pi_1$-injective properly immersed pants in the hyperbolic knot exterior $E(K)$. Then it follows from Agol's result \cite{Ag} that such a $\pi_1$-injective properly immersed pant is either embedded or $E(K)$ can be obtained by Dehn filling one boundary component of the Whitehead link exterior. A knot exterior in $S^3$ cannot contain a properly embedded pant. Therefore one may assume that $E(K)$ is obtained by a Dehn filling of slope $p/q$ along one boundary component of the Whitehead link exterior. Then a homological computation shows that $H_1(E(K); \mathbb{Z}) \cong \mathbb{Z} \oplus \mathbb{Z}/\vert p \vert\mathbb{Z}$, so $\vert p \vert = 1$. Thus $K$ must be a twist knot, which is a 2-bridge knot. Hence this case is also impossible.\end{proof}

Agol's construction \cite{Ag} of a $\pi_1$-injective properly immersed pant in the exterior of the whitehead link shows that meridional tameness never holds for Whitehead doubles of  
non-trivial knots.The exterior of such a knot is obtained by gluing the exterior of a non-trivial knot 
to one boundary component of the Whitehead link exterior. Thus by van Kampen's theorem the $\pi_1$-injectivity of the properly immersed pant in the exterior of the whitehead link is preserved.
The image is a free subgroup generated by two meridians that is not tame as the third boundary component gives rise to another conjugacy class of peripheral elements which is generated by the square of some meridian. 
Note that for non-prime knots meridional tameness does not hold also in general. By \cite{Ag} if one summand is a twist knot (hyperbolic but also the trefoil can occur) then there might exist $\pi_1$-injective properly immersed pants with two boundary components corresponding to a meridian, and thus yielding a non-tame meridional subgroup.


\end{document}